\documentclass[reqno,a4paper,12pt]{amsart} 
\usepackage{amsmath,amscd,amsfonts,amssymb}
\usepackage{mathrsfs,dsfont}
\usepackage{color}
\usepackage{mathtools}
\usepackage{hyperref}
\usepackage{tikz}
\usetikzlibrary{decorations.pathreplacing}
\numberwithin{equation}{section}
\numberwithin{figure}{section}

\def\P{\mathcal{P}}

\def\E{\mathbb{E}}
\def\R{\mathbb{R}}

\def\Q{\mathbb{Q}}
\def\Z{\mathbb{Z}}
\def\N{\mathbb{N}}

\def\B{\mathbb{B}}

\def\T{\mathcal{T}}
\def\S{\mathcal{S}}

\def\1{\mathds{1}}

\renewcommand\leq{\leqslant}
\renewcommand\geq{\geqslant}

\newcommand\Osc{{\operatorname{Osc}}}

\theoremstyle{plain}
\newtheorem{thm}{Theorem}[section]
\newtheorem{lem}[thm]{Lemma}

\newtheorem{prop}[thm]{Proposition}

\newtheorem{conj}[thm]{Conjecture}

\newtheorem*{claim*}{Claim}
\newcommand{\thmref}[1]{Theorem~\ref{#1}}

\newcommand{\conjref}[1]{Conjecture~\ref{#1}}

\theoremstyle{definition}
\newtheorem{definition}[thm]{Definition}
\newtheorem*{definition*}{Definition}
\newtheorem*{remarks*}{Remarks}
\newtheorem*{remark*}{Remark}
\newtheorem{remark}[thm]{Remark}

\usepackage{color}

\newenvironment{enumerate-math}
{\begin{enumerate}
		\addtolength{\itemsep}{5pt}
		}
	{\end{enumerate}}
\newenvironment{enumerate-text}
{\begin{enumerate}
		\addtolength{\itemsep}{5pt}
		}
	{\end{enumerate}}
\newcommand{\m}[1]{\,\left({\rm mod }\ #1\right)}
\newcommand{\G}{\mathcal{G}}
\newcommand{\U}{\mathcal{U}}
\newcommand{\Np}{{\N\setminus p\N}}
\newcommand{\TV}{{\rm TV}}
\newcommand{\ls}{\lesssim}
\newcommand{\p}[1]{\mathbb{P}{\left[#1\right]}}
\newcommand{\e}[1]{\mathbb{E}{\left[#1\right]}}
\renewcommand{\mod}{\,\, {\rm  mod}\,\,}
\newcommand{\la}{\lambda}
\newcommand{\ep}{\varepsilon}
\renewcommand{\L}{\mathcal{L}}

\newcommand{\Pa}{{\rm Pass}}
\newcommand{\X}{\mathcal{X}}
\newcommand{\al}{\alpha}
\newcommand{\be}{\beta}
\newcommand{\ga}{\gamma}
\newcommand{\Y}{\mathcal{Y}}
\newcommand{\hn}{\lfloor n/2 \rfloor}

\newcommand{\D}{\Delta}
\newcommand{\J}{\mathcal{J}}
\renewcommand{\H}{\mathcal{H}}

\newcommand{\tC}{{\widetilde C}}
\newcommand{\I}{\mathcal{I}}

\newcommand{\K}{\mathcal{K}}
\renewcommand{\ll}{\ls}
\renewcommand{\k}{\kappa}
\usepackage{fullpage} 
\usepackage{setspace}
\mathtoolsset{showonlyrefs}

\begin{document}	
	\title[Generalized Collatz]{Generalized Collatz Maps with Almost Bounded Orbits}
	\author{Felipe Gon\c{c}alves}
	\address{Hausdorff Center for Mathematics, 53115 Bonn, Germany.}
	\email{goncalve@math.uni-bonn.de}
	\author{Rachel Greenfeld}
	\address{UCLA Department of Mathematics, Los Angeles, CA 90095-1555.}
	\email{greenfeld.math@gmail.com}
	\author{Jose Madrid}
	\address{UCLA Department of Mathematics, Los Angeles, CA 90095-1555.}
	\email{jmadrid@math.ucla.edu}
	\begin{abstract}
	\emph{If dividing by $p$ is a mistake, multiply by $q$ and translate, and so you'll live to iterate.} We show that if we define a Collatz-like map in this form then, under suitable conditions on $p$ and $q$, almost all orbits of this map attain almost bounded values. This  generalizes a recent breakthrough result of Tao for the original Collatz map (i.e.,  $p=2$ and $q=3$). In other words, given an arbitrary growth function $N\mapsto f(N)$ we show that almost every orbit of such map with input $N$ eventually attains a value smaller than $f(N)$.
	\end{abstract}
	\maketitle

\section{Introduction}	
The \emph{Collatz conjecture} is one of the most challenging problems in Mathematics. This problem concerns the long time behaviours of iterates of the Collatz map $C:\N \to \N$ defined by: $C(N)=N/2$ if $N$ even, $C(N)=3N+1$ if $N$ is odd. 

\begin{conj}[Collatz Conjecture]
For every $N \geq 1$ there is $k$ such that $C^k(N)=1$.
\end{conj}
\noindent For example,
\begin{align*}
&{\rm orbit}(7)=[C^k(7)]_{k\geq 0}=[7, 22, 11, 34, 17, 52, 26, 13, 40, 20, 10, 5, 16, 8, 4, 2, 1, 4, 2,\ldots], \\
& {\rm orbit}(42)=[C^k(42)]_{k\geq 0}=[42, 21, 64, 32, 16, 8, 4, 2, 1,4,2,\ldots]
\end{align*}
and $N=31$ takes $106$ iterations to get to $1$. This conjecture was considered by Collatz around\footnote{The origin of the conjecture is obscure, but Collatz studied very similar maps around the 30's (see Lagarias \cite{L85}).} 1932 and has been verified for all $N\leq 10^{20}$. However, the conjecture survived many attempts of several renowned mathematicians, which made Erd\"os famously pronounce:

\begin{center}
``Mathematics is not yet ready for such problems''.

\hfil Paul Erd\"os
\end{center}

Nevertheless, a lot of work has been produced related to the conjecture\footnote{See, for example, the survey \cite{L85} of Lagarias and some of the bibliography in \cite{L10}.}.  For instance, from the statistical perspective, we have the Benford's Law phenomenon studied by Kontorovich and Miller \cite{KM05} and Lagarias and Soundararajan \cite{LS06} (but other statistical models have also been studied).  From the probability perspective, Krasikov and Lagarias \cite{KL03} showed that
\begin{align}\label{ineq:KL}
\#\left\{N\leq x: T_{\min}(N)=1\right\} \gtrsim \, x^{0.841},
\end{align}
where the exponent is currently best.  Terras \cite{Te} showed that almost all $N\geq 1$ satisfy $C_{\min}(N)<N$, where $C_{\min}(N)=\min_{k\geq 1} C^k(N)$, and Korec \cite{K94} improved this result by showing that for every $\ep>0$ we have $C_{\min}(N) < N^{\log_4 3+\ep}$ for almost all $N\geq 1$.

The recent breakthrough of Tao \cite{T} shows for any given function $f:\N\to\N$ that satisfies $\lim_{N\to \infty} f(N)=\infty$ we have that almost every $N\geq 1$ satisfies $T_{\min}(N)<f(N)$. In other words, Tao shows that the set
\[
\left\{N \in \N : T_{\min}(N)>f(N)\right\}
\]
has zero (logarithmic) density.

\subsection{Collatz-like  maps}
In \cite{C72,C87} Conway studies {generalizations of the Collatz map}. Following Conway, we say that $T:\N\to\N$ is a \emph{Collatz-like} map if $T$ has the form $T(N)=a_N N + b_N$, where $a_N,b_N \in \Q$ are periodic, that is, for some $p$ we have $(a_{N+p},b_{N+p})=(a_N,b_N)$ for all $N\in \N$. Note that additional conditions on $(a_N,b_N)$ have to be assume in order to guarantee that $T$ is globally defined and indeed maps integer to integers. Another equivalent definition for $T$ is to say that there is a partition $\{U_m\}_{m=1}^M$ of the residue classes modulo $p$ such that $T(N)=a_m N + b_m$ if $(N \mod p) \in U_m$. Conway defines the FRACTRAN game (as a wordplay with FORTRAN) and studies its relation with register machines. A FRACTRAN game is a no-players game and it consists of a sequence of positive fractions $F=[f_1,f_2,\ldots,f_M]$, where a partial function $T_F$ is generated by the following rule: Given $N\in \N$ let $T_F(N)$ be the first number $f_m N$ such that $f_m N \in \N$. One could additionally ask that $f_M\in \Z$ to guarantee that $T_F$ is defined in all $\N$. In fact, FRACTRAN games and Collatz-like maps with $b_N\equiv 0$ have a one-to-one correspondence. Indeed, write $f_m=A_m/B_m$ (in lowest terms) and let $p=\mathrm{MCM}(B_1,\ldots,B_M)$. Let now $U_1=\{B_1,2B_1,\dots,{p}\}$ and, in general, $U_{m}=\{B_m,2B_m,\ldots,{p}\}\setminus U_{m-1}$. If $T_F$ is globally defined, then it is easy to see that $U_1,\ldots,U_{M}$ is a partition of $\{1,2,\ldots,p\}$ and that $T_F(N)=NA_m/B_m$ iff $(N \mod p) \in U_m$. Thus $T_F$ is Collatz-like with $b_N\equiv 0$. The converse inclusion is also straightforward.

Conway then defines the famous PRIME game
$$P=\bigg[\frac{17}{91},\frac{78}{85},\frac{19}{51},\frac{23}{38},\frac{29}{33},\frac{77}{29},\frac{95}{23},\frac{77}{19},\frac{1}{17},\frac{11}{13},\frac{13}{11},\frac{15}{14},\frac{15}{2},\frac{55}{1}\bigg]$$ and shows that if one extracts all powers of $2$ contained in the orbit of $N=2^2$ under $T_P$ then this sequence \emph{is exactly} $(2^p)_{p \ \text{prime}}$, that is, the PRIME game lists all primes in increasing order.  Conway also shows that FRACTRAN is a true computer language, in the sense that any Turing machine can be simulated by some FRACTRAN game (and vice-versa obviously). Finally, using the connection with FRACTRAN, Conway shows that  the  question whether, given a FRACTRAN game $F$ and $N\in\N$, the orbit $(T_F^k(N))_{k\geq 0}$ contains a power of $2$, is undecidable (equivalent to the Halting Problem). With a little more work, the statement ``contains a power of $2$" can be replaced by ``contains a number whose prime factorization only has prime powers of the first $M$ primes", where $M$ is fixed {\it a priori}. This is to say that, fixed $M$, there is no global computable function $\mathbb{A}$
$$(T,N)\in \{\text{Collatz-like}
\}\times \N \mapsto \mathbb{A}(T,N) \in  \{0,1\}$$ such that  $\mathbb{A}(T,N)=1$ exactly when there exists $k$ with $T^k(N) = p_1^{a_1}...p_M^{a_M}$ for some $a_1,...,a_M\geq 0$, where $p_1,\ldots,p_M$ are the first $M$ primes.

All this is to say that there is no hope in studying Collatz-like maps $T$ in generality and establishing a (computable) condition that says exactly when such map $T$ has only bounded orbits. We must, therefore, focus our attention to particular subfamilies of Collatz-like maps that are amenable to analytical methods (for instance) and for this reason we restrict ourselves to the following ``simple'' class of maps that essentially mimic the original Collatz map.

	

	
	\begin{definition}[A sub-class of Collatz-like maps\footnote{This class of maps has been previously studied Hasse, M\"{o}ller and Heppner \cite{M78,H79}.}]\label{def:CandS}
	For given integers $p,q\geq 2$ and $r:\Z/p\Z\setminus\{0\}\to \Z$ we let
\begin{align*}
C_{p,q,r}(N)=C(N) := \begin{cases} 
\quad N/p & \text{if} \quad N \equiv 0 \m p, \\
{qN+r(j)} & \text{if} \quad N \equiv j \m p.
\end{cases}
\end{align*}
We additionally require that $qj+r(j)\geq 1$ for all $j=1,\ldots,p-1$, so that $C(\N)\subset \N$.
Respectively, we define the  \emph{speed up} (or Syracuse) map, $S_{p,q,r}=S:\Np\to\Np$, by
$$
S(N):=\frac{qN+r(j)}{p^m},
$$
where $j=N\mod p$, $m=\nu_p(qN+r(j))$, $\nu_p$ is the $p$-adic valuation and $\N\setminus p\N$ is the set of natural numbers not divisible by $p$.
\end{definition}

Note that we \emph{do not} require $p$ or $q$ to be prime, however such restriction can be produce nice examples. The original Collatz map is recovered here by taking $p=2$, $q=3$ and $r(1)=1$. 

Some of these maps are in the literature already, for instance, the $Qx+1$ maps (where $Q$ is some odd number)  were studied in \cite{BMS,KL,S1,S2,V}. However these are still instances of Collatz-like maps $T$ where the function $N\mapsto (a_N,b_N)$ is $2$-periodic. To the best of our knowledge, this work is the first to show analytic results for maps in such generality.

\subsection{Main result} 

Conway's undecidability result implies that there are Collatz-like maps for which the question whether all orbits are bounded is undecidable. It follows that in order to solve the Collatz conjecture one must use some special properties of the Collatz map.  The focus of this work is on the problem of determining these special properties.

Despite Conway's result, we managed to generalize 
the recent breakthrough of Tao \cite{T}, on the boundedness of \emph{almost all}  Collatz orbits, to a wide class of Collatz-like maps.

\begin{thm}\label{thm:main}
Let $p,q$ and $r:\Z/p\Z\setminus\{0\}\to \Z$ be given and consider the maps $C$ and $S$ as in Definition \ref{def:CandS}. Assume further that:
\begin{enumerate}
\item[(a)] $p$ and $q$ are coprime; 
\item[(b)] $q<p^{{p}/(p-1)}$; 
\item[(c)] $qj+r(j) \equiv 0 \ ({\rm mod}\ p)$.
\end{enumerate}
Then almost all orbits of $S$ and $C$ attain almost bounded values in the following sense: For every $\delta>0$ there exists $B_\delta>0$ such that
$$
\sum_{N \in [1,x] \setminus p\N \,:\, S_{\min}(N) < B_{\delta}} \frac1N > (1-\delta) \tfrac{p-1}{p}\log x \quad \text{and} \quad \sum_{N \in [1,x] \,:\, C_{\min}(N) < B_{\delta}} \frac1N > (1-\delta) \log x
$$
for $x$ sufficiently large, where $S_{\min}(N)=\min_{k\geq 0} \{S^k(N)\}$ and $C_{\min}(N)=\min_{k\geq 0} \{C^k(N)\}$. Equivalently, for any function $f:\N\to\N$ with $\lim_{N\to\infty} f(N)=\infty$ the sets 
\begin{equation}\label{eq:logbdd}
    \{N \in \Np \,:\, S_{\min}(N)  > f(N)\} \quad \text{and} \quad \{N \in \N \,:\, C_{\min}(N) > f(N)\}
\end{equation}
have zero logarithmic density.
\end{thm}

Since $C(N)=S(N/p^{\nu_p(N)})$, it is not hard to show that the result for the $S$ map implies the result for the $C$ map and vice-versa. 

\begin{remark}
One of the major difficulties in generalizing Tao's result was formulating the correct statement for Theorem \ref{thm:main}, i.e., to identify the correct class of ``Collatz-like'' maps to which we could apply various techniques, similar to those introduced in \cite{T}. The second obstacle was to extract the results about the ``triangle-avoiding'' two-dimensional renewal process embedded in Tao's proof (see \cite[Section 7]{T})  and make it modular (in the proper sense), so we could apply it as a black box to our setting. Moreover, we had to overcome the fact that when $q$ is not a prime, triangles of the ``bad'' region (see Steps 1 and 2 in Section \ref{sec:proofmain}) behave a bit differently close to the edge of the region in Theorem \ref{thm:renewal} compared to Tao's original argument.
\end{remark}

We believe that the conditions (a), (b), and (c) in Theorem \ref{thm:main} are in fact sufficient to guarantee that the existence of $B_C>0$ such that $C_{\min}(N) \leq B_C$ for \emph{all} $N\in \N$, as our following conjecture asserts\footnote{A similar conjecture was previously implicitly suggested in \cite{M78,H79}, without the condition that $p$ and $q$ are coprime.}.

\begin{conj}\label{conj:main}
The maps $S$ and $C$ as in Definition \ref{def:CandS} have only finitely many cycles and all its orbits eventually enter in one of these cycles if all conditions in Theorem \ref{thm:main} are satisfied.
\end{conj}

In Section \ref{numerics} we present numerical evidences which support Conjecture \ref{conj:main} and also give examples showing that the converse of the conjecture is actually false. 
However, we were not able to  find maps $C$  such that all orbits seemed to be bounded but $q\geq p^{p/(p-1)}$. Thus, it seems that  condition (b) must be satisfied whenever all orbits are bounded, and so one can regard it as the most crucial condition.

\begin{remark}[The Matthews \& Watts class \cite{MW84}]
The class of maps we consider in Definition \ref{def:CandS} was in fact considered before (in more generality) by Matthews \& Watts in \cite{MW84}, however we follow the survey of Matthews \cite{M10} for details. For $p\geq 2$ and integers $(b_m)_{m=0}^{p-1}$, $(q_m)_{m=0}^{p-1}$ they define the map
\begin{align*}
&C(N)=(q_mN +b_m)/p \ \ \ \text{if } \ (N \text{ mod } p) = m.
\end{align*}
We require that $q_m>0$ and $q_m m  + b_m \in  p\N$ for all $m=0,...,p-1$ so that $C$ is well-defined and $C(\N)\subset \N$, although Matthews \& Watts originally considered these maps over $\Z$ and did not require such conditions. The map $C$ is said to be of relatively prime type if $\gcd(q_m,p)=1$ for all $m$. Observe that our Definition \ref{def:CandS} is (almost) a particular instance of their class (besides diving by $p$ for $ N \text{ mod } p \neq 0$), and it is realized by taking $q_0=1$, $r_0=0$, $q_j=q$ and $b_j=r(j)$ for $j\geq 1$. Based on some numerical evidence, Matthews \& Watts pose the following conjecture (see \cite[Conjecture 3.1]{M10}), which is much stronger and general than our Conjecture \ref{conj:main}, and reads as follows:  \emph{Let $C$ be of relatively prime type. Then the number of cycles is finite and unbounded trajectories $\{C^{k}(N)\}_{k\geq 0}$ are uniformly distributed module $p^a$ for any $a\geq 1$ (in the natural density sense). If $q_0\dots q_{p-1}<p^p$ then all trajectories are bounded, and so eventually enter in a cycle. If $q_0\dots q_{p-1}>p^p$ then almost all trajectories are divergent. }
Note that the equality case $q_0\dots q_{p-1}=p^p$ never occurs due to the primality constraints.

Let $q^{(k)}(N)=q_{m}$ and $b^{(k)}(N)=b_m$ if $C^{(k)}(N) \equiv m \ (\text{mod } p)$. Then it is easy to see that (see \cite[Theorem 2.1 (ii)]{M10})
\begin{align}\label{eq:MWforumla}
\begin{split}
C^{(k)}(N) & =N \frac{q^{(0)}\cdots q^{(k-1)}(N)}{p^k}\prod_{\ell=0}^{k-1} \left(1 - \frac{r^{(\ell)}(N)}{q^{(\ell)}(N)C^{(\ell)}(N)}\right) \\ & 
= N \left[\prod_{m=0}^{p-1} \bigg(\frac{q_m}{p}\bigg)^{\la^{(k)}_m}\right]^k \prod_{\ell=0}^{k-1} \left(1 - \frac{r^{(\ell)}(N)}{q^{(\ell)}(N)C^{(\ell)}(N)}\right)
\end{split}
\end{align}
where $\la^{(k)}_m(N)=$ frequency of $q_m$ in $(q^{(0)}(N),q^{(1)}(N),...,q^{(k-1)}(N))$.
The heuristic here is that the frequencies $\la^{(k)}_m(N)$ are uniformly distributed modulo $p$ and so we expect $\la^{(k)}_m(N)\to \tfrac1p$ as $k\to\infty$. Then, if the second product in \eqref{eq:MWforumla} is $p^{o(k)}$  we  would have
\[
\lim_{k\to\infty} (C^{(k)}(N))^{1/k} =  \frac{\sqrt[p]{q_0\dots q_{p-1}}}{p}
\]
and so we derive the condition  $q_0\dots q_{p-1}<p^p$. Moreover, they show in \cite[Theorem 1]{MW84} that if $\{C^{k}(N)\}_{k\geq 0}$ is unbounded and uniformly distributed modulo $p$ then the limit above holds. 

As a final comment we note that the map $C$ can be easily seen as a transformation in the $p$-adic integers $\Z_p$ and it is not hard to show that it preserves the Haar measure in $\Z_p$. Matthews \& Watts prove (see \cite[Corollary 1]{MW84}) that in fact $C$ is strongly mixing, a fact which was known already for the original Collatz map \cite[Theorem K]{L85}.
\end{remark}

\begin{remark}
We now describe an interesting numerical feature we have found, which we could not find in the literature. Consider the $3X-1$ map: $C(N)=N/2$ if $N$ is even, otherwise $C(N)=(3N-1)/2$ (this is equivalent to the original Collatz map, but over the negative integers). It is conjecture that all cycles of $C$ are $\{[1], [5, 7, 10], [17, 25, 37, 55, 82, 41, \\ 61, 91, 136, 68, 34]\}$, and that every orbit eventually enters in one of these cycles. Then one can ask the question: \emph{Which numbers $N$ contain $1$ in its orbit?} We can ask the same question for the numbers $a=5,7,10,17,\ldots,34$, and more generally for any other integer $a\in \N$. What we found numerically is that if we let $\omega(a)=\cup_{k=0}^\infty C^{-k}(a)$ then it seems that 
$$
\# (\omega(a) \cap [1,x] ) \sim \kappa(a) x,
$$
for some constant $\k(a)\geq 0$. Indeed, we often see that the plot of the first $10^4$ terms of the sequence $\omega(a)$ looks very much like a line already and this is true for most $a$'s less than $10^3$. We also have the estimates
\begin{align*}
& \k(1)\approx 0.33, \ \ \k([5, 7, 10])\approx [0.0, 0.27, 0.055], \\
& \k([17,25,\ldots,34])\approx [0.0, 0.055, \ldots, 0.044]
\end{align*}
and also that $\k(1) \approx \k(5) +\k(7) + \k(10) \approx \k(17)+\k(25)+\ldots +\k(34) \approx 1/3$, that is, no cycle attracts more integers than others. We observed similar phenomena for all other maps in the Matthews \& Watts class we tested.

For instance, for the original Collatz map (and also other maps), one question that would be interesting to explore is whether we indeed have
$$
\p{\exists n \geq 0 \text{ s.t. } a\in S^n(\L_{x,x^\al})} \sim \k(a) x,
$$
where $\L_{x,x^\al}$ is a logarithmic-distributed random variable in $[x,x^\al] \cap \Np$ and $S$ is the Syracuse-map (see Section \ref{sec:stabfirstpass}), as such a  statement is consistent with the above phenomena. Another interesting question is that, if we let $S^\N(N)=[S^{n}(N)]_{n\geq 0}$ be the orbit of $N$, then we can view $S^\N(\L_{x,x^\al})$ as a random variable in $\{0,1\}^\N$ and we can ask if the following limit
$$
\lim_{x\to\infty} S^\N(\L_{x,x^\al})
$$
exists (in the weak*-sense) as a random variable in $\{0,1\}^\N$ and what properties this limiting random variable possess. It may be even possible to use the total variation bound Proposition \ref{prop:firstpassage} to prove this. However, these are questions that we leave for future work.

\end{remark}

\subsection{Logarithmic density vs. natural density}
 The \emph{natural density} (or asymptotic density) of a set $A\subset \N$ is usually defined by the limit
\[
\text{dens}(A)=\lim_{N\to\infty} \frac{\#(A\cap \{1,\ldots,N\})}{N},
\]
whenever it exists, whereas the log-density of $A$ is defined by
\[
\text{logdens}(A)=\lim_{N\to\infty} \frac1{\log N}{\sum_{\substack{M\leq N \\ M\in A}}\frac1M}.
\]
Our main result, Theorem \ref{thm:main}, concerns the \emph{logarithmic density} of the orbits, which is a weaker notion than the more common \emph{natural density}.
Indeed, is easy to show that if $\text{dens}(A)=1$ then $\text{logdens}(A)=1$, but the converse is false. 
It is therefore natural to aim to strengthen  \thmref{thm:main} to a statement about ``almost all'' orbits in the natural density sense, but this seems a very hard task and still is out of reach.  However, a strengthening of Krasikov and Lagarias result \eqref{ineq:KL} seems a natural question to explore in the log-density case and perhaps the techniques of our paper (and Tao's paper \cite{T}) can be adapted to show that under the conditions of Theorem \ref{thm:main} there is $B>0$ such that
$$
\sum_{\substack{M\leq N \\ C_{\min}(N)\leq B}} \frac{1}{M} \gtrsim_\ep (\log N)^{1-\ep}
$$
for every $\ep>0$. We leave this question for future work. Nevertheless, despite being unable to deal with natural density for generic growth functions $N\mapsto f(N)$, we were able to generalize a weaker natural density result of Korec \cite{K94}\footnote{A weaker version of this was previously obtained by Allouche in \cite{A79}} to the class of maps $C$ in Definition \ref{def:CandS} when $f(N)=N^a$.

\begin{thm}\label{prop:Korec}
Let $C$ be as in Definition \ref{def:CandS} and assume all conditions of Theorem \ref{thm:main}. Assume further that $r(j) \geq 0$ for all $j$. Let $\ga=\tfrac{p-1}{p}\log_p q$. Then for every $c>0$ the set 
$$
\left\{N\in \N \, :  \, C_{\min}(N)<N^{\ga+c}\right\}
$$
has natural density $1$.
\end{thm}

\subsection{A two dimensional renewal process} It turns out that (as in Tao \cite{T}) the main key ingredient in the proof of Theorem \ref{thm:main} is the study of a random walk on $\N^2$ with positive increments. We state the result here as it can be of independent interest. Let $\X_1=(\J,\L)$ be a random variable with values in $\N^2$ and consider the random walk
\begin{align}\label{def:renewalproc}
\X_{1,k}=\sum_{j=1}^k \X_j
\end{align}
where $\X_j$ are i.i.d. copies of $\X_1$. Such a random walk is called a (arithmetic) \emph{renewal process}. We define
$$
\omega=\frac{\e{\L}}{\e{\J}},
$$
so the expected slope of the increment in this renewal process is $\omega$. One way of interpreting the process $\X_{1,k}$ is of a tourist wondering around a city, with a tendency of walking in a certain direction given by $\e{\X}$. Suppose that there are some bad neighbourhoods $\B \subset  [n]\times \N$ (where $[n]=\{1,2,\ldots,n\}$) in the city that the tourist would like to avoid during the visit. We then want to bound the total expected time the tourist spends in these bad neighbourhoods. We will assume that the bad set
$\B$ has the following properties: 
\begin{enumerate}
\item[(i)] $\B \subset \sqcup \D(j_\D,l_\D,s_\D)$
is a disjoint union of triangles of the form
\begin{align}\label{def:triangles}
\D(j_\D,l_\D,s_\D) := \{(j,l)\in \N^2 : j \geq j_\D, l\leq l_\D, \eta(j -j_\D) + (l_\D-l) \leq s_\D\},
\end{align}
where $\eta>0$ is some fixed slope independent of the triangles; 
\item[(ii)] Triangles are uniformly separated from each other and from the edge $\{n\}\times \N$, i.e., 
\[
{\rm dist}(\D,\Delta') \gtrsim \log(1/\ep) \quad \text{and} \quad {\rm dist}(\D,\{j=n\})\gtrsim \log(1/\ep)
\]
for any distinct $\D,\D'\subset \B$ (for some small $\ep>0$).
\item[(iii)] We typically exit a triangle from the top, that is, $\omega> \eta$.
\end{enumerate}
We have the following result.

\begin{thm}\label{thm:renewal}
Let $\X_{1,k}$ be defined as in \eqref{def:renewalproc}. Assume $\X$ is not supported in a coset of a proper subgroup of $\Z^2$ and that for some $c>0$
$$
\p{|\X|>t} \ll e^{-ct}
$$
for all $t>0$. Assume also conditions {\rm (i)}, {\rm (ii)} and {\rm (iii)} for the bad set $\B$, for some $\ep>0$. Then for every $A>0$ we have
\[
\e{\#\{k\in [n] : \X_{1,k} \in \B\}} \leq n-A\log n + O_{A,\ep}(1).
\]
In fact, we have the stronger estimate
$$
\e{\exp\left({-\ep \#\{k\in [n] : \X_{1,k} \notin \B\}}\right)} \ls_{A,\ep} n^{-A}
$$
(which implies the previous bound by Jensen's inequality).
\end{thm}

This is result is implicit in Tao's paper (see \cite[Proposition 1.17 \& Eq. (7.8)]{T}) and it is proven only when $\X$ is geometrically distributed.



\subsection{Organization of the paper}
In Section \ref{sec:prelim} we prove Theorem \ref{prop:Korec}. In Sections  \ref{sec:distorbits}, \ref{sec:finescale} and \ref{sec:equiv} we perform a series of reductions for our main result and we finally prove it in Section \ref{sec:proofmain}. In Section \ref{sec:proofrenewal} we prove Theorem \ref{thm:renewal}, which is used in Section \ref{sec:proofmain} as a black box. In Section \ref{numerics} we explore numerically the maps $C$ and  \conjref{conj:main}. 

\subsection{Acknowledgments}
RG was partially supported by the Eric and Wendy Schmidt Postdoctoral Award. We would like to  thank Terry Tao for helpful discussions. We also  thank the Hausdorff Research Institute for Mathematics in Bonn, and the organizers of the trimester program ``Harmonic Analysis and Analytic Number Theory'', during which a significant portion of this work was conducted. We are also grateful to the anonymous referee for several suggestions that improved the exposition of this paper and for  alerting us to the relevant references \cite{M78,H79}.

\section{Preliminaries}\label{sec:prelim}

\subsection{Notation}\label{sec:notation}
We describe here some important information about the notation we will use throughout the paper without further mention.
\begin{itemize}
    \item The letter $c$ will be reserved to represent a generic sufficiently small positive quantity that can change from line to line. 
    \item We write $\N=\{1,2,3,\ldots\}$ for the natural numbers, $\Z=\{\ldots,-1,0,1,\ldots\}$ for the integers and  $\Np$ for the natural numbers not divisible by $p$. 
    \item For a given $x\in \Z$ we use the notation $x \mod N$ to denote the residue class of $x$ as an element of $\Z/N\Z$ and we write $x\equiv y \m{N}$ to indicate that $x-y$ is divisible by $N$. 
    \item For an infinite (or finite) vector $v=(v_1,v_2,\ldots)$ we write $v^{(m,n)}=(v_m,v_2,\ldots,v_n)$ and let $v^{(n)}=v^{(1,n)}$.
    \item We denote partial sums as
$$
v_{m,n}=\sum_{i=m}^n v_i.
$$
We use $|\cdot|$ for the $\ell^1$-norm $|v|=\sum_{i} |v_i|$, and in particular $|v^{(m,n)}|=v_{m,n}$ if $v^{(m,n)}$ has nonnegative coordinates.
\item For the rest of the paper we will set $\mu=p/(p-1)$. 
\item We write $A\ls B$ to mean there is a positive constant $K>0$ such that $A\leq K B$ (a priori, all implicit constants $K$ are allowed to depend on $p,q,r$).
\item We use $A\approx B$ when $A\ls B$ and $B\ls A$. We use $A\ls_{\la_1,\la_2,\ldots} B$ to emphasize the implicit constant $K$ may depend on other parameters $\la_1,\la_2,\ldots$.
 Statements such as $A\ls B \Rightarrow C\ls D$ mean that if $A\leq K B$ for some $K>0$ then there is $K'=K'(K)$ such that $C\leq K' D$. 
 \item Random variables shall be denote by calligraphic notation: $\X,\Y,\G,\P$ etc. If $\X$ is a random variable we always use $\X_j$ to denote i.i.d. copies of $\X$ and we write $\X^{(n)}=(\X_1,\X_2,\ldots, \X_n)$.
 \item If $\X$ and $\Y$ are random variables we use $\X\equiv \Y$ to say that they have the same distribution. We also define their total variation by
\begin{align}\label{eq:TVdef}
\TV[\X,\Y] = \sup_{E\subset \Omega} |\p{\X\in E} - \p{\Y\in E}| \approx_{\text{if } \Omega \text{ is discrete}} \sum_{\omega\in \Omega} |\p{\X=\omega} - \p{\Y=\omega}|.
\end{align}
\item For $\nu>0$ we let $\G(\nu)$ be the geometrical distribution on $\N$ with average $\nu$
$$
\p{\G(\nu)=N}:= \frac{1}{\nu}\left(\frac{\nu-1}{\nu}\right)^{N-1}
$$
and $\P(\nu):=\G(\nu)+ \G(\nu)$ be the Pascal distribution
$$
\p{\P(\nu)=N}=\frac{N-1}{\nu^2}\left(\frac{\nu-1}{\nu}\right)^{N-2}.
$$
\item If $F$ is a finite set we denote by $\U(F)$ the uniform distribution on $F$. 
\item We will use logical connectors $\wedge$ (\emph{and}) and $\vee$ (\emph{or}) when dealing with probabilistic events $E$ and we will use the bar notation $\overline{E}$ to denote the event \emph{not} $E$. 
\end{itemize}

\subsection{Reduction to the case $p\geq q$ and $\gcd(q,r(1),\ldots,r({p-1}))=1$}\label{sec:firstreduction}

Indeed, if $p>q$ then
$C^2(N) \leq \frac{p-1}{p}N+A$, where $A=\max_j |r(j)|$, and we obtain that $C^{2k}(N) < ({(p-1)}/{p})^{k}N + pA$ for all $k$. This shows that $C_{\min}(N)\leq pA$ for all $N\geq 1$.
Moreover, without loss of generality we can assume
\begin{align}\label{eq:rjassumption}
h=\gcd(q,r(1),\ldots,r({p-1}))=1.
\end{align}
Otherwise, if $h>1$, since $\gcd(p,q)=1$, then $h$ is invertible modulo $p$ and thus we can define a function $r^\star:\Z/p\Z\setminus\{0\}\to \Z$ by $r^\star(j ( h^{-1} \mod p)) := r(j)/h$. The new map $C^\star$ with the same $p$ and $q$ now satisfies the conditions of Definition \ref{def:CandS} and  we have $C(hN)=hC^\star(N)$ (and so $C$ and $C^\star$ are topologically conjugated in $h\N$). Noting that $\tfrac{N}B \leq C(N) \leq B N$ for some $B>1$ and all $N\geq 1$, it is then not hard to show that
\[
C_{\min}(N)\approx h C^\star_{\min}(C(Np^{-\nu_p(N)})/h).
\]
We obtain
\begin{align*}
\sum_{L \in [1,x] \,:\, C_{\min}(L) > f(L)} \frac{1}{L} & = \sum_{j\geq 0} p^{-j} \sum_{\substack{N \in [1,xp^{-j}] \,:\, C_{\min}( N) >f(p^jN) \\ \nu_p(N)=0}} \frac{1}{N} \\
& \leq \sum_{j\geq 0} p^{-j} \sum_{\substack{N \in [1,xp^{-j}] \,:\, C^\star_{min}(C(N)/h) \gtrsim f(p^j N) \\ \nu_p(N)=0}} \frac{1}{N} \\
& \ls  \sum_{j\geq 0} p^{-j} \sum_{M\leq Bxp^{-j}/h \,:\,  C^\star_{min}(M) \gtrsim f_j^{\star}(M)}\frac{1}{M},
\end{align*}
where in the last inequality we set $M=C(N)/h$ and $f_j^{\star}(M)=\min\{f(p^jN): C(N)/h=M\}$. Note that since there are at most $p-1$ integers $N$ that satisfy $C(N)=hM$ and these solutions belong to the interval $[hM/B,hMB]$, we conclude $\lim_{M\to\infty} f^{\star}_j(M)=\infty$. Using that $\max(1,\log(x)) \gtrsim \max(1,\log (Bxp^{-j}/h))$,
is not hard to see that we can apply Theorem \ref{thm:main} in conjunction with the dominated convergence theorem to obtain that 
\begin{align*}
& \limsup_{x\to\infty} \frac{1}{\max(1,\log x)}\sum_{L \in [1,x] \,:\, C_{\min}(L) > f(L)} \frac{1}{L} =0.
\end{align*}

\subsection{Proof of Theorem \ref{prop:Korec}}
We will now show that for every $c>0$
\[
\text{dens} \left\{N\in \N \, :  \, C_{\min}(N)<N^{\ga+c}\right\} =1,
\]
where $\ga=\tfrac{p-1}{p}\log_p q$. Note that $q<p^{p/(p-1)}$ if and only if $\ga<1$. Clearly, we can work  instead with the map $\tC(N)=N/p$ if $p$ divides $N$ and $\tC(N)=(qN+r(j))/p$ otherwise, because $\tC_{\min}(N)=C_{\min}(N)$.  Let
$X=\left\{N\in \N \, :  \, \tC_{\min}(N)<N^{\ga+c}\right\}$ for $c>0$,
$ \chi_j(N)=\tC^{j}(N) \mod p$, $\chi^{(k)}(N)=(\chi_{0}(N),\chi_{1}(N),\dots,\chi_{k-1}(N))$ and 
$
U(m,\delta)=\{N\in \N :  \#\{k\leq m-1 : \chi_k(N)\neq 0\} \leq \delta m\}.
$
It follows from Lemma \ref{lem:SRperiodic} that $\chi^{(m)}(M)=\chi^{(m)}(N)$ if and only if $N \equiv M\m{p^m}$, thus the set $\{\chi^{(m)}(N)\}_{N=M}^{p^m+M-1}$ lists all vectors of $\{0,\ldots,p-1\}^m$ for any $M\in \N$. In particular
\[
\#U(m,\delta)\cap [M,p^m+M)=\sum_{k=0}^{\lfloor m\delta \rfloor} {{m}\choose{k}}(p-1)^k,
\]
hence, by the central limit theorem, we have
$\#U(m,\delta)\cap [M,p^m+M)=(1+o_m(1))p^m$ for any $\delta>\frac{p-1}{p}$. The idea here is the following: for a given $x\geq 1$ we select $m=m(x)$ such that $(m-1)^2p^{m-1}<x\leq m^2p^m$ and observe that
\begin{align*}
\{N\in X : mp^m \leq N \leq x\} 
 \supset \bigcup_{j=m+1}^{ \lfloor xp^{-m}\rfloor} \{N\in X : (j-1)p^m \leq N < jp^m\}.
\end{align*}
Select now $\delta$ such that $\frac{p-1}{p}<\delta<\frac{(\ga+c)}{\log_p q}$. We claim that $U(m,\delta) \cap [mp^m,x] \subset X$ if $m$ is sufficiently large. Assuming this claim is true we conclude that 
\[
\#\{N\in X : mp^m \leq N \leq x\} \geq  \sum_{j=m+1}^{ \lfloor xp^{-m}\rfloor}\#U(m,\delta)\cap [(j-1)p^m,jp^m) = (1+o_m(1))x,
\]
and this clearly shows that $X$ has natural density $1$. To prove the claim we use that
$N/p \leq \tC(N) \leq (qN  + A)/p$, where $A=\max_j r(j)$. Assuming that $N\in U(m,\delta) \cap [mp^m,x]$ and letting $a=\#\{k\leq m-1 : \chi_k(N)\neq 0\}\leq \delta m$ we obtain
\begin{align*}
    \tC^{m}(N)=N\prod_{i=0}^{m-1}\frac{\tC(\tC^i(N))}{\tC^{i}(N)}\leq N\left(\frac{qm+A}{pm}\right)^a\left(\frac{1}{p}\right)^{m-a}
    \leq \frac{Nq^{\delta m}}{p^m}\left(1+\frac{A}{qm}\right)^{\delta m},
\end{align*}
where in the first inequality above we have used that $\tC^i(N)\geq N/p^i\geq m$ (since $r(j)\geq 0$) and $\frac{q\tC^i(N)+A}{p\tC^i(N)}\leq \frac{qm+A}{pm}$. The inequality $ \frac{Nq^{\delta m}}{p^m}\left(1+\frac{A}{qm}\right)^{\delta m} < N^{\ga+c}$ now follows by the bounds $mp^m\leq N \leq m^2 p^m$ and the choice of $\delta$. \hfill $\qed$

\section{Distribution of the valuation of the Syracuse orbits}\label{sec:distorbits}
For a given $N\in \N$ we let
$$ a(N):=(a_i(N))_{i\geq 1}=\nu_p(qS^{i-1}(N)+R_{i}(N)))_{i\geq 1}
$$
where 
$$
R(N) := (R_i(N))_{i\geq 1}= (r({S^{i-1}(N) \mod  p}))_{i\geq 1}. 
$$ 
Note $a_i(N)\geq 1$ for all $i$ by the assumption on the function $r$, which guarantees that $qS^{i-1}(N)+R_{i}(N)$ is divisible by $p$ at least once. As in \cite[Heuristic 1.8]{T}, we expect that for a typical large natural number $N$ the entries of the vector $a^{(n)}(N)$, for $n\ls \log N$, behaves like i.i.d. copies of a geometrical distribution over $\N$. Specifically, if $\X$ is a random variable with values in $\Np$ and uniformly distributed $\m {p^m}$ then we expect
\begin{align}\label{ineq:aapprox}
a^{(n)}(\X) \approx \G(\mu)^{(n)} \quad \text{for } n\gg m,
\end{align}
where\footnote{We reserve $\mu$ to always denote $p/(p-1)$.}
\[
\mu:=p/(p-1)
\]
and $\G(\mu)$ is a geometrical random variable taking values in $\N$ and with mean $\mu$, that is,
$$
\p{\G(\mu)=N} = (p-1)p^{-N}.
$$
We make this heuristic exact in Proposition \ref{prop:heuristic}. In fact, this is closely related with the Benford’s Law phenomenon studied in \cite{KM05,LS06}. A straightforward induction computation proves the very important identity
\begin{align}\label{eq:Siterates}
\begin{split}
S^{n}(N)&= q^{n} p^{-a_{1,n}(N)}  N + \sum_{i=1}^n q^{n-i}p^{-a_{i,n}(N)}R_{i}(N)
\\ & = q^{n}  p^{-a_{1,n}(N)} N + F_n(a^{(n)}(N),R^{(n)}(N))
\end{split}
\end{align}
for all $N,n\in \N$, where
\begin{align}\label{eq: Fk formula}
 F_n (a,R) :=\sum_{i=1}^n q^{n-i}p^{-a_{i,n}(N)}R_{i}(N)=\sum_{i=1}^n q^{i-1} p^{-a_{n-i+1,n}}R_{n-i+1} 
\end{align}
for\footnote{With abuse of notation, we will use the letters $a$ and $R$ to stand by generic vectors while the functions $N\mapsto a(N)$ and $N\mapsto R(N)$ will always be written with the argument $N$.} $(a,R)\in \N^n \times \{r(1),\ldots,r(p-1)\}^n$. A key property of $F$ to be used later, is the formula
\begin{align}\label{eq:Fnrecursionformula}
 F_n (a^{(1,n)},R^{(1,n)})=F_{n-i}(a^{(i+1,n)},R^{(i+1,n)})+q^{n-i} p^{-a_{i+1,n}}F_{i}(a^{(i)},R^{(i)})
\end{align}
for $i\leq n$. In particular, by the law of large numbers and formulas \eqref{eq:Siterates} and \eqref{ineq:aapprox}, the following size approximation should typically hold
$$
S^n(\X)\approx q^np^{-|a^{(n)}(\X)|}\X\approx q^np^{-|\G(\mu)^{(n)}|}\X \approx (q\,p^{-p/(p-1)})^n \X.
$$ 
That is, the remainder term in the iteration formula \eqref{eq:Siterates} should not affect the long time growth behaviour of a typical orbit of $S$. This heuristics suggests the critical threshold \[
q<p^{\tfrac{p}{p-1}}
\]
of Conjecture \ref{conj:main}. The following proposition corroborates this heuristic rigorously.

\begin{prop}\label{prop:heuristic}
Let $\X$ be a random variable taking values in $\Np$ and assume 
$$
\TV{\left[\X  \mod {p^m},\U(\Np  \mod {p^m})\right]} \ls p^{-m}
$$
for some $m\in \N$. If $n\in\N$ is such that $m\geq (\mu+c')n$ for some $c'>0$ then there is $c=c(c',p)>0$ such that
\begin{align}\label{ineq:aGeom}
\TV{\left[a^{(n)}(\X),\G(\mu)^{(n)}\right]}\ls  \, p^{-c n}.
\end{align}
\end{prop}
\noindent We note that in practice we will take $c'\geq \mu$.
We will need the following lemmas to prove this proposition (which are similar to the ones used by Terras \cite{Te}).

\begin{lem}\label{lem:SRperiodic}
Let $N\in\Np$ and $m=|a^{(k)}(N)|$. Then $R^{(k+1)}(M)=R^{(k+1)}(N)$ and $a^{(k)}(M)=a^{(k)}(N)$ if $M\equiv N\,\m {p^{m+1}}$ .
\end{lem}

\begin{proof}
Let $m=|a^{(k)}(N)|$. In fact, we shall prove by induction that\footnote{Here $a_{k+1,k}(N):= 0$.}
\begin{align}\label{eq:claimiteratesS}
\begin{cases}&S^i(M)  \equiv  S^i(N)\,  \m {p^{a_{i+1,k}(N)+1}} \\ 
&R_i(M)=R_i(N)\\
&a_i(M)=a_i(N)
\end{cases}
\end{align}
for $i=1,\ldots,k$ if $M\equiv N\,\m {p^{m+1}}$. The base case $i=1$ is simple and we leave the details to the reader. Assume \eqref{eq:claimiteratesS} is true for some $i<k$.  We have
$$
R_{i+1}(M) = r_{S^{i}(M) \mod  p} = r_{S^{i}(N) \mod p} = R_{i+1}(N).
$$
Since $S^{i}(M)=S^{i}(N)+\lambda p^{a_{i+1,k}(N)+1}$ and  $a_{i+1,k}(N)\geq a_{i+1}(N)=\nu_p\bigl(qS^{i}(N)+R_{i+1}(N)\bigr)$, we obtain
$$
a_{i+1}(M)=\nu_p\bigl(qS^{i}(M)+R_{i+1}(M)\bigr) = \nu_p\bigl(qS^{i}(N)+R_i(N)+\lambda p^{a_{i+1,k}(N)+1} \bigr) = a_{i+1}(N)
$$
and
$$
S^{i+1}(M)=p^{-a_{i+1}(M)}\bigl(qS^{i}(M)+R_{i+1}(M)\bigr) =S^{i+1}(N)+q\lambda  p^{a_{i+2,k}(N)+1}.
$$
Finally, for $i=k$ we see that since  $S^{k}(M) \equiv S^k(N)\m  p$ then $R_{k+1}(M)=R_{k+1}(M)$. The lemma is therefore proven.
\end{proof}

\begin{lem}\label{lem:Vsetsingleton}
The set $E_k(a,R)=\{N\in \Np : a^{(k)}(N)=a, R^{(k+1)}(N)=R\}$ is exactly one residue class modulo ${p^{|a|+1}}$ for any given $a\in \N^{k}$ and $R\in \{r(1),\ldots,r(p-1)\}^{k+1}$.
\end{lem}
\begin{proof}
By the iteration formula \eqref{eq:Siterates} (for $n$ replaced by $k+1$) we see that if $N\in E_{k}(a,R)$  then
$$
N  \equiv  - (q^{-k-1} \mod p^{|a|+1}) \sum_{i=1}^{k+1} q^{k+1-i}p^{a_{1,i-1}}R_{i}\quad  \m {p^{|a|+1}}.
$$
Also, by Lemma \ref{lem:SRperiodic}, if $M\equiv N \m {p^{|a|+1}}$ then $a^{(k)}(M)=a^{(k)}(N)$ and $R^{(k+1)}(M)=R^{(k+1)}(N)$, hence $M\in E_k(a,R)$. We conclude that it is enough to show that $E_k(a,R)$ is non-empty, and we now prove this by induction on $k$. If we let $E_0(a,R)=E_0(R)$ then it is easy to see
$E_0(r_j)=\{j+\gamma p: \gamma\in\N\}$, and therefore the base case of induction can be regarded as $k=0$. Assume the lemma is true for some $k\geq 0$. Let $a\in\N^{k+1}$ and $R=(r_{j_1},\ldots,r_{j_{k+2}})$, and write $R^*=(r_{j_2},\ldots,r_{j_{k+2}})$ and $a^*=(a_2,\ldots,a_{k+1})$. By the induction hypothesis
$$
E_{k}(a^*,R^*)=\{M + \gamma p^{|a^*|+1}: \gamma \in \N\}
$$
for some $M\in \Np$.  Select $\theta,\alpha \in \N$ so that $p^{a_{1}-1}\theta\equiv  1 \m q$ and $\theta (qj_1 +r_{j_1})/p + \alpha q \equiv M \m {p^{|a^*|+1}}$. These selections are possible because $p$ and $q$ are coprime and $qj+r_j$ is divisible by $p$ by definition. Observe now that 
$$
p^{a_1}(\theta (qj_1 +r_{j_1})/p + \alpha q)  =q(j_1+p\gamma)+r_{j_1}
$$
for some $\gamma\in\N$. Therefore, if we let $M'=\theta (qj_1 +r_{j_1})/p + \alpha q$ and $N=j_1+p\gamma$ we see that $M\in E_{k}(a^*,R^*)$ and $S(N)=M'$. We conclude  $N\in E_{k+1}(a,R)$, which completes the proof of  the lemma.
\end{proof}

\begin{lem}\label{lem:stirling}
Let $m,k\in \N$ be such that $m\geq (\mu+c') k$. Then there is $c=c(c',p)>0$ such that
$$
(p-1)^k p^{-m}\binom{m}{k} \ls p^{-c m}.
$$
\end{lem}
\begin{proof}
Let $k=um$ for $0<u\leq 1/(\mu+c')$ and $H(x)=-x\log x-(1-x)\log (1-x)$ be the entropy function. Stirling's approximation gives
\begin{align*}
(p-1)^k p^{-m}\binom{m}{k} & \approx p^{-(1-u \log(p-1)/\log(p) )m} \frac{m^{m+1/2}}{((1-u )m)^{(1-u)m+1/2} (um)^{um+1/2}} \\
& = p^{-m(1-u\log(p-1)/\log(p) -H(u)/\log(p))}(mu(1-u))^{-1/2}.
\end{align*}
Since the function $f(u)=1-u\log(p-1)/\log(p) -H(u)/\log(p)$ is convex, $f(0)=1$ and $f(1/\mu)=f'(1/\mu)=0$, we conclude that $f(u)$ is decreasing and positive for $u\in (0,1/(\mu+c'))$. We deduce that
$$
p^{-m(1-u\log(p-1)/\log(p) -H(u)/\log(p))}(\mu(1-u))^{-1/2} \ls p^{-f(1/(\mu+c')) m}.
$$
We can then take $c=f(1/(\mu+c'))$ to finish the proof.
\end{proof}

\begin{proof}[Proof or Proposition \ref{prop:heuristic}]
We start first by proving inequality \eqref{ineq:aGeom}. We claim that 
\begin{align}\label{ineq:atail}
\p{a_{1,n}(\X)\geq m}\ls p^{-c m}.
\end{align}
for some $c>0$. Since $\p{a_{1,n}(\X)\geq m}= \sum_{k=1}^{n} \p{a_{1,k-1}(\X)<m\leq a_{1,k}(\X)}$ (letting $a_{1,0}(\X)=0$) and $n\ls m$, it is enough to show that $\p{a_{1,k-1}(\X)<m\leq a_{1,k}(\X)}\ls p^{-c m}$ uniformly for any $k=1,\ldots,n$ (and then replace $c$ by $c/2$ to take care of the linear factor). Explicitly we have
\begin{align}\label{eq:a1kbetweenm}
\p{a_{1,k-1}(\X)<m\leq a_{1,k}(\X)} = \sum_{\substack{a\in \N^{k} \, : \, a_{1,k-1}<m\leq a_{1,k} \\ \ell \in \{1,\ldots,p^m\} \setminus p\N}} \p{a^{(k)}(\X)=a, \, \X\equiv \ell \, \m {p^m}}.
\end{align}
Because $N\mapsto R_i(N)$ attains at most $p-1$ values,  the iteration formula \eqref{eq:Siterates} implies that under the constraint $a^{(k)}(N)=a$ there at most $(p-1)^k$ possible values for $N \m {p^m}$ if $a_{1,k-1}<m\leq a_{1,k}$. We conclude
\begin{align*}
\p{a_{1,k-1}(\X)<m\leq a_{1,k}(\X)} &  \ls \sum_{{a\in \N^{k} \, : \, a_{1,k-1}<m\leq a_{1,k} }} (p-1)^kp^{-m}   \leq (p-1)^kp^{-m}\binom{m}{k}  \ls p^{-c m},
\end{align*}
where we have applied Lemma \ref{lem:stirling} in the last inequality. This proves the claim \eqref{ineq:atail}.  Janson's bound \cite{J} states that
\begin{align}\label{ineq:jansontail}
\p{|\G(\mu)^{(n)}|\geq t n \mu} \leq e^{-n(t-1-\log t)}
\end{align}
for any $t\geq 1$. Since $m\geq (\mu+c')n$, we can use this bound for $t=(\mu+c')/\mu$ to obtain
\begin{align}\label{ineq:geomtail}
\p{|\G(\mu)^{(n)}|\geq m} \leq p^{-c n},
\end{align}
for some $c>0$. We obtain
\begin{align}\label{ineq:TVfirstestimate}
\begin{split}
\TV[\G(\mu)^{(n)},a^{(n)}(\X)] & = \sum_{a\in \N^n \,:\, |a|<m} \bigl|(p-1)^n p^{-|a|}-\p{a^{(n)}(\X)=a}\bigr| + O(p^{-c n}),
\end{split}
\end{align}
where the big-$O$ above comes from the sum over $\{a\in \N^n \,:\, |a|\geq m\}$ and the tail bounds \eqref{ineq:atail} and \eqref{ineq:geomtail}.  Applying now Lemma \ref{lem:Vsetsingleton} we conclude
\begin{align*}
\p{a^{(n)}(\X)=a}
& = \sum_{R \, \in \{r(1),\ldots,r(p-1)\}^{n+1}} \p{\X \in E_n(a,R)}\\
& = \sum_{R\,\in \{r(1),\ldots,r(p-1)\}^{n+1}} \p{ \left(\X \mod p^{|a|+1}\right) \in \left(E_n(a,R) \mod p^{|a|+1}\right)} \\
& = \sum_{R\,\in \{r(1),\ldots,r(p-1)\}^{n+1}} p^{-|a|}/(p-1) + O(p^{-m}) \\
& = (p-1)^n p^{-|a|} + O(p^{-m}(p-1)^{n})
\end{align*}
Finally we deduce that
\begin{align*}
\TV[\G(\mu)^{(n)},a^{(n)}(\X)] & \ls p^{-c n} + \sum_{a\in \N^n \,:\, |a|<m} p^{-m}(p-1)^{n}\\
& = p^{-c n} +  p^{-m}(p-1)^{n}\binom{m-1}{n}\\
& \ls p^{-c n},
\end{align*}
where we again have used Lemma \ref{lem:stirling} in the last inequality.
\end{proof}

\subsection{Stabilisation of first passage}\label{sec:stabfirstpass}

For integers $x < y$ we let
$
\L_{x,y}
$
be a random variable with logarithm distribution and supported on $\Np \, \cap [x,y)$, that is,
$$
\p{\L_{x,y} = N} = \frac{1}N\bigg(\sum_{M\in \Np \, \cap [x,y)}\frac{1}M\bigg)^{-1}
$$
for all $N\in \Np \, \cap [x,y)$.
Define the first time passage by
$$
T_{x}(N):=\inf\{n\geq 0 : S^{n}(N)\leq x\},
$$
with $\inf\{\emptyset\}=\infty$. Note that if $T_x(N)=n$ then $T_x(S^{n-m}(N))=m$, but conversely, if $T_x(S^{n-m}(N))=m$ and $S^k(N)>x$ for $0\leq k\leq n-m$ then $T_x(N)=n$. We will use this property later on. We also define the first passage by \[
\Pa_x(N):=S^{T_{x}(N)}(N),
\]
with $S^\infty(N)=1$.

	\begin{prop}\label{prop:firstpassage}
There is $c>0$ such that, if $\al,\be\in(1,1+c)$,  then we have
\begin{align}\label{eq:passage time estimate}
\p{T_{x}(\L_{x^\be,x^{\al\be}})=+\infty}\lesssim x^{-c}
\end{align}
and
\begin{align}\label{eq:first passage location estimate}
\TV{\left[\Pa_x(\L_{x^\be,x^{\al\be}}),\Pa_x(\L_{x^{\al\be},x^{\al^2\be}})\right]}\lesssim \log^{-c}x,
\end{align}
for all $x\geq 2$
\end{prop}

\begin{proof}[{\bf Proof that Proposition \ref{prop:firstpassage} $\Rightarrow$ Theorem \ref{thm:main}}]
Replacing $\beta=\alpha,\alpha^2$, an argument similar to the one presented in \cite[Section 3]{T} can be used to show that Theorem \ref{thm:main} is a consequence of Proposition \ref{prop:firstpassage}. We leave the details to the reader.
\end{proof}

	\section{Fine scale mixing}\label{sec:finescale}
	We model the remainder term in the iteration formula \eqref{eq:Siterates} by the $\Z_q$-valued random variable
	\begin{equation}\label{eq:syracRV}
	\S = \sum_{i=1}^\infty q^{i-1} p^{-\G_{1,i}}\U_i
	\end{equation}
	where the convergence of the sums above is to be taken in the $q$-adic $\Z_q$, while $(\G_1,\G_2,\ldots)$ and $(\U_1,\U_2,\ldots)$ are vectors with i.i.d. copies of $\G(\mu)$ and $\U({r_1,\ldots,r_{p-1}})$ respectively. We then define its projections by 
	\begin{equation}\label{eq:syracRVproj}
	\S_n := \S \mod q^n \equiv \sum_{i=1}^n q^{i-1} p^{-\G_{1,i}}\U_i \mod q^n.
	\end{equation}
Observe now that by \eqref{eq: Fk formula} we have that 
\begin{align}\label{eq:Snidentities}
\S_n \equiv F_n(\widetilde\G^{(n)},\widetilde\U^{(n)}) \mod q^n \equiv F_n(\G^{(n)},\U^{(n)})  \mod q^n
\end{align}
where $\widetilde\G^{(n)}$ is $\G^{(n)}$ in reverse order and similarly $\widetilde\U^{(n)}$.
		\begin{prop}\label{prop:oscillation}
	For any  integers $1\leq m<n$ and $A>0$ we have
	\begin{align}\label{eq: osc}
	\Osc_{m,n}(\p{S_n=N})_{N\in \Z/q^n\Z}&:=\sum_{N\in \Z/q^n \Z}\left| \p{\S_n=N} - \frac1{q^{n-m}}\sum_{M\in \Z/q^n \Z \,:\, M=N \mod q^m} \p{\S_n=M}\right|\nonumber\\ &\ls_A m^{-A}.
	\end{align} 
 \end{prop}
 
\noindent This proposition can be interpreted in the following way: When we consider the vector $V=(\p{\S_n=N})_{N\in \Z/q^n \Z} \in \Q^{q^n}$, then its $\ell^1$-distance to the projection of $V$ into subspace of vectors of $\Q^{q^n}$ which are $q^m$-periodic is very small, in other words, the vector $V$ does not oscillate too much in finer scales. Also note that since $
\S_m = \S_n \mod q^m$ we have
\begin{align}\label{eq:syracmoduloidentity}
\p{\S_m=N \mod q^m} = \sum_{M\in \Z/q^n \Z \,:\, M=N \mod q^m} \p{\S_n=M}.
\end{align}

\begin{proof}[{\bf Proof that Proposition \ref{prop:oscillation} $\Rightarrow$ Proposition \ref{prop:firstpassage}}]
$\nonumber$
Throughout the whole proof we let $x\geq 2$ and $n_0\in\N$ be such that $x^{\al} \approx p^{6\mu n_0}$, that is, 
\[
n_0 = \lfloor \frac{\al}{6\mu\log p} \log x \rfloor.
\]
We highlight that several claims made below will be true once $x$ is sufficiently large and so we will avoid repeating this phrase constantly.

\noindent {\bf Step 1.}
A direct computation using the asymptotic formula 
\begin{align}\label{eq:logasymp}
\sum_{N\in \Np \, \cap [x,y)} \frac1N =(1+O(1/x)) \frac1{\mu} \log(y/x)
\end{align}
shows that for all $M\in \Z/p^k\Z$ not divisible by $p$ we have
\[
\p{\L_{x^\al,x^{\al\be}} \mod p^k=M}=(p-1)p^{-k}+O(p^{-2k})
\]
whenever $p^{3k} \ls x^{\al\be}$. We conclude that 
$$
\TV{\left[\L_{x^\be,x^{\al\be}}  \mod {p^{k_0}},\U(\Np  \mod {p^{k_0}})\right]} \ls p^{-k_0}
$$
with $k_0=\lfloor 2\mu n_0 \be\rfloor$. We can then apply Proposition \ref{prop:heuristic} to get
\begin{equation}\label{ineq:tvaandgeom}
\TV{\left[a^{(n)}(\L_{x^\be,x^{\al\be}}),\G(\mu)^{n}\right]}\ls  \, p^{-c n},
\end{equation}
for any $n\leq n_0$. Fix $\gamma:=\tfrac12{\log_p q}+\tfrac12\mu$ such that ${\log_p q}<\gamma <\mu$. We obtain
\begin{align}\label{ineq:taila}
\begin{split}
\p{|a^{(n_0)}(\L_{x^\be,x^{\al\be}})|\leq \gamma n_0}
&\leq \p{|\G(\mu)^{n_0}|\leq \gamma n_0}+O(p^{-c n_0})\\
& \leq \p{|n_0\mu-|\G(\mu)^{n_0}||\geq (\mu-\gamma) n_0}+O(p^{-c n_0})\\
&\lesssim p^{- c n_0}
\\ & \ls x^{-c}
\end{split}
\end{align}
where above we applied \cite[Lemma 2.2]{T} to obtain the following Chernoff type tail bound
\begin{align}\label{ineq:chernofftail}
\p{||\G(\mu)^{(n)}|-n\mu|\geq t} & \ls e^{-ct^2/n}+e^{-ct} \quad (t>0).
\end{align}
From the iteration formula \eqref{eq:Siterates}, the event $|a^{(n_0)}(\L_{x^\be,x^{\al\be}})|>\gamma n_0$ implies that
\begin{align*}
S^{n}(\L_{x^\be,x^{\al\be}}) \lesssim q^{n_0}p^{-|a^{(n_0)}(\L_{x^\be,x^{\al\be}})|}x^{\al\be}+ q^{n_0} \ls x^{\frac{\al({\log_p q}-\gamma)}{6\mu}+\be} + x^{\frac{\al{\log_p q}}{6\mu}} \ls x^{1-c},
\end{align*}
if $\be,\al>1$ are sufficiently close to $1$. We conclude that the event $|a^{(n_0)}(\L_{x^\be,x^{\al\be}})|>\gamma n_0$ implies that $S^{n}(\L_{x^\be,x^{\al\be}}) \leq x$  for $x$ sufficiently large, and in this case we have $T_x(\L_{x^\be,x^{\al\be}})\leq n<+\infty$. We conclude that
$$
\p{T_{x}(\L_{x^\be,x^{\al\be}})=+\infty}\leq \p{|a^{(n_0)}(\L_{x^\be,x^{\al\be}})|\leq \gamma n_0}  \lesssim x^{-c}.
$$
\medskip

\noindent {\bf Step 2.}
In order to establish \eqref{eq:first passage location estimate} it is enough to prove that 
\begin{equation}\label{eq: pass goal}
    \p{\Pa_x(\L_{x^\be,x^{\al\be}})\in E}=Q_{x,\al,E}+O(\log^{-c}x)
\end{equation}
for any $E\subset [1,x] \cap \Np$, where $Q_{x,\al,E}$ is some quantity independent of $\be$ (as long as $\al,\be>1$ are sufficiently close to $1$). Indeed, we can then replace $\be$ by $\al\be$ and in this way derive \eqref{eq:first passage location estimate}. Let $\L=\L_{x^\be,x^{\al\be}}$ and
\begin{equation}\label{a-mu n<log 0.6 x}
A^{k}=\{a\in \N^k : |a_{1,n}-\mu  n|<\log^{0.5 + \ep} x \ \text{for all} \ n \ \text{with} \ 0\leq n\leq k\}.
\end{equation}
By \eqref{ineq:tvaandgeom} and the tail bound \eqref{ineq:chernofftail}  we get
\begin{equation}\label{eq: error for not in A}
\p{a^{(n_0)}(\L)\notin A^{n_0}}=\p{\G(\mu)^{n_0}\notin A^{n_0}} + O(p^{-cn_0})  \ls \log^{-c} x.
\end{equation}
and therefore
\begin{align}\label{ineq:decaynotinA}
\p{{\Pa}_{x}(\L)\in E}=\p{{\Pa}_{x}(\L)\in E  \wedge a^{(n_0)}(\L)\in A^{n_0}}+O(\log^{-c}x).
\end{align}
Assume now that $a^{(n_0)}(\L)\in A^{n_0}$. Then for all $n$ with $0\leq n\leq n_0$ we have
\begin{align}\label{ineq:Sestimate2}
\begin{split}
S^n(\L) =q^{n}p^{-a_{1,n}(\L)}\L+O(q^{n_0}).& =(1+O(x^{-c}))q^{n}p^{-a_{1,n}(\L)} \L  \\ & = e^{O(\log^{0.6}x)}\left(\frac{q}{p^\mu}\right)^n\L,
\end{split}
\end{align}
which implies that
\begin{align}\label{ineq:Sestimate3}
e^{-\log^{0.7} x} \left(\frac{q}{p^\mu}\right)^{n-m} S^m(\L) \leq S^n(\L)  \leq e^{\log^{0.7} x} \left(\frac{q}{p^\mu}\right)^{n-m} S^m(\L)
\end{align}
for all $0\leq m,n\leq n_0$. Observe that since $q/p^\mu<1$ we have $e^{O(\log^{0.6}x)}\left(\frac{q}{p^\mu}\right)^{n_0}\L \leq x$ (for large $x$) and so $T_x(\L)\leq n_0$.
From the definition of $T_x(\L)$ we deduce
\[
\exp\left(O(\log^{0.6}x)\right)\left(\frac{q}{p^\mu}\right)^{T_x(\L)}\L \leq x < \exp\left(O(\log^{0.6}x)\right)\left(\frac{q}{p^\mu}\right)^{T_x(\L)-1}\L,
\]
or, equivalently,
\begin{equation*}
T_x(\L)=\frac{\log(\L/x)}{\log(p^{\mu}/q)}+O(\log^{0.6}x).
\end{equation*}
We then consider the interval
$$
I_y=\left[\frac{\log(y/x)}{\log\frac{p^{\mu}}{q}}-\log^{0.7}x,\frac{\log(y^{\alpha}/x)}{\log\frac{p^{\mu}}{q}}+\log^{0.7}x\right] \cap (\Np)
$$
with $y=x^\be$ and note that
\begin{equation}\label{eq: size of I_y}
\#I_y=(1+O(\log^{-c}x))\frac{\alpha-1}{\mu \log \frac{p^{\mu}}{q}}\log y.
\end{equation}
One can show that $
\p{T_{x}(\L)\in I_y \wedge a^{(n_0)}(\L)\in A^{n_0}}=1-O(\log^{-c}x)$,
which implies
\begin{align*}
\p{{\Pa}_{x}(\L)\in E} & =\p{T_{x}(\L)\in I_y \wedge {\Pa}_{x}(\L)\in E\wedge a^{(n_0)}(\L)\in A^{n_0}}+O(\log^{-c}x)
\\ & = \sum_{n\in I_y} \p{B_n} + +O(\log^{-c}x)
\end{align*}
where for a given $n\in I_y$ we define the event
$$
B_n:=[T_x(\L)=n \wedge \Pa_{x}(\L)\in E\wedge a^{(n_0)}(\L)\in A^{n_0}].
$$
Let
\begin{align}\label{eq: definition of m_0}
m_0:=\lfloor{\frac{\alpha-1}{100\mu \log p}\log x\rfloor}
\end{align}
and note $I_y\subset [2m_0,n_0]$ (any $m_0\approx \ep n_0$ would suffice). We claim that 
\[
B_n = [T_x(S^{n-m_0}(\L))=m_0\wedge \Pa_{x}(S^{n-m_0}(\L))\in E\wedge a^{(n_0)}(\L)\in A^{n_0}].
\]
The claim follows at once if we show $S^k(\L)>x$ for $0\leq k\leq n-m_0$. Indeed, the event $T_x(S^{n-m_0}(\L))=m_0$ implies that $S^n(\L)\leq x< S^{n-1}(\L)$, which in conjunction with $a^{(n_0)}(\L)\in A^{n_0}$, allow us to use \eqref{ineq:Sestimate3} to obtain
\begin{align*}
S^k(\L)  > e^{-\log^{0.7} x} \left(\frac{p^\mu}{q}\right)^{n-1-k} S^{n-1}(\L) 
  > e^{-\log^{0.8} x}\left(\frac{p^\mu}{q}\right)^{m_0} x  > x^{1+c} e^{-\log^{0.8} x} >x.
\end{align*}
This proves the claim. We can then apply \eqref{ineq:Sestimate3} to deduce
\begin{align}\label{eq: formula for P(B_n)}
    \p{B_n}=\p{S^{n-m_0}(\L)\in E'\wedge a^{(n_0)}(\L)\in A^{n_0}}
\end{align}
where 
\[
E' = \left\{M\in \N\setminus{p\N} :  T_x(M)=m_0, \ \Pa_x(M)\in  E\right\}.
\] 
Since the event $a^{(n_0)}(\L)\in A^{n_0}$ is contained in the event $a^{(n-m_0)}(\L)\in A^{n-m_0}$ we obtain
\begin{align}\label{eq: formula for pass depending on Aff_a,r}
\begin{split}
& \p{\Pa_x(\L)\in E} \\ & =\sum_{n\in I_y}    \p{S^{n-m_0}(\L)\in E'\wedge a^{(n-m_0)}(\L)\in A^{n-m_0}}+O(\log^{-c}x)\\
&=\sum_{\substack{n\in I_y\\a\in A^{n-m_0}\\ R\in \{r(1),\dots,r(p-1)\}^{n-m_0}\\ M\in E'}}\p{R^{(n-m_0)}(\L)=R \wedge S^{n-m_0}(\L)=M \wedge a^{(n-m_0)}(\L)=a}+O(\log^{-c}x).
\end{split}
\end{align}
It is not hard to show (see for instance \cite[Lemma 2.1]{T}) that if for some $M\in \Np$, $R\in \{r(1),\ldots,r(p-1)\}^k$ and $a\in \N^k$ we have $M\equiv F_k(a,R) \m {q^{k}}$ then
\[
N:=p^{|a|}\frac{M-F_{k}(a,R)}{q^{k}} \in \Np,
\]
and moreover, $a^{(k)}(N)=a$, $R^{(k)}(N)=R$ and $S^{k}(N)=M$. Using this observation in conjunction with \eqref{eq:logasymp} we obtain
\begin{align*}
& \p{\Pa_x(\L)\in E} \\ & =\sum_{n\in I_y}    \p{S^{n-m_0}(\L)\in E'\wedge a^{(n-m_0)}(\L)\in A^{n-m_0}}+O(\log^{-c}x)\\
&=\sum_{\substack{n\in I_y\\a\in A^{n-m_0}\\ R\in \{r(1),\dots,r(p-1)\}^{n-m_0}\\ M\in E' \\ M\equiv F_{n-m_0}(a,R) \m {q^{n-m_0}}}}\frac{1}{(1+O(1/x))\frac{\alpha-1}{\mu}\log y}\, \frac{q^{n-m_0}p^{-|a|}}{M-F_{n-m_0}(a,R)}+O(\log^{-c}x) \\
& = \sum_{\substack{n\in I_y\\a\in A^{n-m_0}\\ R\in \{r(1),\dots,r(p-1)\}^{n-m_0}\\ M\in E' \\ M\equiv F_{n-m_0}(a,R) \m {q^{n-m_0}}}}\frac{(1+O(x^{-c}))}{\frac{\alpha-1}{\mu}\log y}\frac{p^{-|a|}q^{n-m_0}}{M} + O(\log^{-c}x).
\end{align*}
In the third identity above we have used that $M-F_{n-m_0}(a,R)=M-O(q^{n_0})=(1+O(x^{-c}))M$, which follows because if $B_n$ holds and  $M=S^{n-m_0}(\L)$ then $T_x(M)=n$ and by \eqref{ineq:Sestimate2} we have
$$
M=e^{O(\log^{0.6} x)} (q/p^\mu)^{n-m_0} \L = e^{O(\log^{0.6} x)} (p^\mu/q)^{m_0} x \geq x^{1+c}.
$$
 By the size estimate \eqref{eq: size of I_y} of $I_y$, it is now enough to show that
\[
q^{n-m_0} \sum_{\substack{a\in A^{n-m_0}\\ R\in \{r(1),\dots,r(p-1)\}^{n-m_0}\\ M\in E' \\ M\equiv F_{n-m_0}(a,R) \m {q^{n-m_0}}}}\frac{p^{-|a|}}{M} = Q'_{x,\al,E} + O(\log^{-c} x)
\]
where $Q'_{x,\al,E'}$ is independent of $n$ and $\be$. Now observe that by \eqref{eq:syracRVproj} the quantity above can be rewritten as
$$
\E[1_{\G^{(n-m_0)}\in A^{n-m_0}}C_{n}(\S_{n-m_0})],
$$
$\G=(\G_1,\G_2,\ldots)$ and $\U=(\U_1,\U_2,\ldots)$ is are infinite vectors with i.i.d. copies of $\G(\mu)$ and $\U({r_1,\ldots,r_{p-1}})$ respectively, and $C_n:\Z/q^{n-m_0}\Z\to\Q_{+}$ is the function defined by
$$
C_{n}(N):=q^{n-m_0}\sum_{M\in E'\,:\,M=N \mod q^{n-m_0}}\frac{1}{M}.
$$
We claim that $C_{n}(N)=O(1)$ for $n\geq m_0$ and $N\in \Z/q^{n-m_0}\Z$, and we postpone the proof to Step $3$. Using \eqref{eq: error for not in A}, Proposition \ref{prop:oscillation} and \eqref{eq:syracmoduloidentity} we then conclude
\begin{align*}
 \E[1_{\G^{(n-m_0)}\in A^{n-m_0}}C_{n}(\S_{n-m_0})] 
& = \E[C_{n}(\S_{n-m_0})]+O(\log^{-c}x) \\
& = \sum_{N\in \Z/q^{n-m_0}\Z}C_n(N)\p{\S_{n-m_0}=N}+O(\log^{-c}x) \\
& =  \sum_{N\in \Z/q^{n-m_0}\Z}C_n(N)q^{2m_0-n}\p{\S_{m_0}=N \mod q^{m_0}}+O(\log^{-c}x)\nonumber\\
&=q^{m_0}\sum_{M\in E'}\frac{1}{M}\p{\S_{m_0}=M \mod q^{m_0}}+O(\log^{-c}x)\nonumber\\
&=Q'_{x,\al,E}+O(\log^{-c}x),
\end{align*}
where $Q'_{x,\al,E'}$ is independent of $n$ and $\be$.
\medskip

\noindent {\bf Step 3.} To complete the proof it remains to show the claim that $C_n(N)=O(1)$ for $n\geq m_0$ and $N\in \Z/q^{n-m_0}\Z$. First we write
\begin{align}\label{eq:Cnexpanded}
C_n(N)=\sum_{\substack{a\in \N^{n-m_0} \\ R
\in \{r(1),\ldots,r(p-1)\}^{n-m_0+1}}}q^{n-m_0} \sum_{\substack{M\in E' \\ M=N \mod q^{n-m_0} \\ a^{(n-m_0)}(M)=a \\ R^{(n-m_0+1)}(M)=R}}\frac{1}{M}.
\end{align}
Let $M$ satisfy the constraints in the inner sum above. Since for $M\in E'$ we have $T_x{M}=m_0$ and 
\[
F_{n}(a,R) = O(q^{m_0})  \ls x^{c},
\]
for any $n\leq m_0$, we deduce that
\begin{align}\label{ineq:rangeforM}
x \approx q^{m_0}p^{-|a|}M.
\end{align}
Moreover, by the iteration formula \eqref{eq:Siterates}, we also have
$$
q^{m_0}M+p^{|a|}F_{n}(a,R))\equiv R_{n-m_0+1}p^{|a|} \m {p^{|a|+1}},
$$
and in particular, applying the Chinese remainder theorem, $M$ is constrained to one specific residue class $N'(a,R) \mod q^{n-m_0}p^{|a|+1}$. We can then use \eqref{ineq:rangeforM} to estimate the inner sum of \eqref{eq:Cnexpanded} to obtain
\[
q^{n-m_0}\sum_{\substack{M\in E' \\ M=N \mod q^{n-m_0} \\ a^{(n-m_0)}(M)=a \\ R^{(n-m_0+1)}(M)=R}}\frac{1}{M} \ls \frac{q^{n-m_0} }{xq^{-m_0}p^{|a|}} \sum_{\substack{M\approx xq^{-m_0}p^{|a|} \\ M\equiv N'(a,R) \m{ q^{n-m_0}p^{|a|+1} }}} 1  \ls p^{-|a|}.
\]
We conclude 
\[
C_n(N) \ls \sum_{\substack{a\in \N^{n-m_0} \\ R
\in \{r(1),\ldots,r(p-1)\}^{n-m_0+1}}}  p^{-|a|} = (p-1),
\]
and this finishes the proof.
	\end{proof}
	
\section{Equivalence to decay of Fourier coefficients}\label{sec:equiv}
In this section we stablish an equivalent form of Proposition \ref{prop:oscillation}.





\begin{prop}\label{prop:decaycharacteristic}
Let $n\in\N$ and $\xi\in\Z/q^n\Z$ not divisible by $q$. Then for every $A>0$ we have that
$$
\E[e^{-2\pi i \xi \S_n/q^n}] \ls_A n^{-A}.
$$
\end{prop}

\begin{proof}[{\bf Proof that Proposition \ref{prop:decaycharacteristic} $\Rightarrow$ Proposition \ref{prop:oscillation}}] $\nonumber$

\noindent {\bf Step 1.} Let  $\gamma:=\tfrac12 \frac{\log_p q}{\mu} + \tfrac12$
so that $\frac{\log_p q}{\mu} < \ga < 1$.
It suffices to prove Proposition \ref{prop:oscillation} in the case
\begin{equation}\label{eta n}
 \gamma n \leq m < n.
\end{equation}
Indeed, once we have \eqref{eq: osc} in this case, since $
\S_m = \S_n \mod q^m$, we also obtain (by periodicity)
\[
\sum_{N \in \Z/q^{m'}\Z} \left|q^{m-m'} \p{\S_m=N \mod q^{m'}} - \p{ \S_{m'}  = N} \right| \ll_A m^{-A}
\]
whenever $\gamma n \leq m \leq m' \leq n$, and the general case follows from telescoping sums and the triangle inequality.
We  fix $A\geq 1$ and let $C_A > 0$ be a sufficiently large constant to be chosen later. We may also assume that $A$ is large and $n$ is sufficiently large depending on $A$. Let $(\G_1,\G_2,\dots,\G_n) \equiv \G(\mu)^{(n)}$. For $k\leq n$ define the event
\begin{equation}\label{ij}
E_k:=\bigwedge_{1\leq i\leq j\leq k}\left( |\G_{i,j} - \mu (j-i+1)| \leq C_A ( \sqrt{(j-i+1)(\log n)} + \log n ) \right).
\end{equation}
Observe $E_n$ implies $E_k$. We can apply the tail bound \eqref{ineq:chernofftail} to obtain (denoting by $\overline E_n$ the complement of $E_n$)
\begin{align}
\p{\overline{E_n}} & \leq \sum_{1\leq i\leq j \leq n}\p{|\G_{i,j} - \mu (j-i+1)| > C_A ( \sqrt{(j-i+1)(\log n)} + \log n )} \\ & \ls n^2 e^{-cC_A \log n}\\ 
&  \ls n^{-A}
\end{align}
if $C_A$ is large.
Note that if $E_n$ holds then
$$ \G_{1,n} \geq \mu n - C_A (\sqrt{n \log n} + \log n) > n {\log_p q}$$
for $n \geq n_A$, since by assumption ${\log_p q} < \mu$. Thus, under the event $E_n$ there is a unique natural number $k < n$ such that
\[ 
\G_{1,k} \leq n {\log_p q} - C_A^3 \log n < \G_{1,k+1},
\]
Such $k$ must satisfy (if $A$ and $C_A$ are large enough)
\begin{equation}\label{bmd-1}
|k - n \frac{\log_p q}{\mu}|  \leq C_A^A( \log n ).
\end{equation}
and, if $ \G_{1,k+1}=M$,
\begin{equation}\label{bmd-2}
n {\log_p q} - C_A^3 \log n \leq  M \leq n {\log_p q} - \tfrac{1}{2} C_A^3 \log n.
\end{equation}
By \eqref{ij} and the triangle inequality it will suffice to show that
\[
\sup_{\substack{k \text{ in the range } \eqref{bmd-1} \\ M \text{ in the range } \eqref{bmd-2}}} \Osc_{m,n}\left( \p{ \S_n = N \wedge E_{k+1} \wedge G_{1,k+1}=M } \right)_{N \in \Z/q^n\Z} \ll_{A} n^{-2A}.
\]
Letting $g = g_{n,k,M}\colon \Z/q^n\Z \to \R$ denote the function
\[
g_{n,k,M}(N) :=\p{ \S_n = N \wedge E_{k+1} \wedge G_{1,k+1}=M }
\]
we want to show that
$$ 
 \sum_{N \in \Z/q^n\Z} \left|g(N) - q^{m-n} \sum_{M \in \Z/q^n\Z \,:\, M = N \mod q^m} g(M) \right| \ll_{A} n^{-2A}
$$
for $\ga n\leq m\leq n$, $k$ in the range \eqref{bmd-1} and $M$ in the range \eqref{bmd-2}. By Cauchy-Schwarz, it is sufficient to show that
\begin{equation}\label{half}
q^n\sum_{N \in \Z/q^n\Z} \left|g(N) - q^{m-n} \sum_{M \in \Z/q^n\Z: M = N \mod q^m} g(M) \right|^2 \ll_{A}  n^{-4A}.
\end{equation}
By Fourier inversion and Plancherel's theorem, the left hand side of \eqref{half} is equal to
\[
\sum_{\xi \in \Z/q^n\Z \,:\, q^{n-m} \,\nmid \,\xi} \left| \sum_{N \in \Z/q^n\Z} g(N) e^{-2\pi i \xi N / q^n} \right|^2=\sum_{\xi \in \Z/q^n\Z \,:\, q^{n-m} \,\nmid \,\xi} \left| \e{ e^{-2\pi i \xi \S_n / q^n} \1_{E_{k+1} \wedge \G_{1,k+1}=M}} \right|^2.
\]

Observe now that by \eqref{eq:Fnrecursionformula} (for $i$ replaced by $n-k-1$), under the event $\G_{1,k+1}=M$, we have the key identity
\begin{align*}
\S_n & \equiv F_{k+1}(\G^{(k)},\U^{(k)}) + q^{k+1} p^{-M} F_{n-k-1}(\G^{(k+2,n)},\U^{(k+2,n)}) \mod q^n \\
& \equiv \S_{k+1} + q^{k+1} p^{-M} \S'_{n-k-1} \mod q^n,
\end{align*}
where $\S'_{n-k-1}$ is a copy of $\S_{n-k-1}$ independent of $\S_{k+1}$. The fundamental point here is that under the event $\G_{1,k+1}=M$ the two terms in the above expression are independent. Letting $\xi'=\xi p^{-M} \mod q^{n-k-1}$, since $q \nmid \xi$ we deduce that $q\nmid \xi'$, hence we can apply Proposition \ref{prop:decaycharacteristic} for some $A'>0$ to obtain
\begin{align*}
& \sum_{\xi \in \Z/q^n\Z \,:\, q^{n-m} \,\nmid \,\xi} \left| \e{ e^{-2\pi i \xi \S_n / q^n} \1_{E_{k+1} \wedge \G_{1,k+1}=M}} \right|^2 \\ 
& = \sum_{\xi \in \Z/q^n\Z \,:\, q^{n-m} \,\nmid \,\xi} \left| \e{ e^{-2\pi i \xi \S_{k+1} / q^n} \1_{E_{k+1} \wedge \G_{1,k+1}=M}} \e{ e^{-2\pi i \xi' \S_{n-k-1} / q^{n-k-1}}}\right|^2\\
& \ls_{A'} (n-k-1)^{-2A'} \sum_{\xi \in \Z/q^n\Z \,:\, q^{n-m} \,\nmid \,\xi} \left| \e{ e^{-2\pi i \xi \S_{k+1} / q^n} \1_{E_{k+1} \wedge \G_{1,k+1}=M}} \right|^2 \\
& \ls_{A'}  n^{-2A'} q^n \sum_{N \in \Z/q^n\Z} \p{\S_{k+1}=N \wedge E_{k+1} \wedge \G_{1,k+1}=M}^2,
\end{align*}
where in the last inequality we have used that $n-k-1 \geq n(1-\tfrac{\log_p q}{\mu}) + O_A(\log n) \gtrsim n$ and applied Plancherel's theorem after discarding the condition $q^{n-m} \,\nmid \,\xi$. The last term in the chain of inequalities above is equal to
\begin{align*}
n^{-2A'} q^n \sum_{\substack{N \in \Z/q^n\Z \\ a\in \N^{k+1} \\ R\in \{r(1),\dots,r(p-1)\}^{k+1}}} \p{ \U^{(k+1)}=R \wedge \G^{(k+1)}=a \wedge \S_{k+1}=N \wedge E_{k+1} \wedge \G_{1,k+1}=M}^2.
\end{align*}
We now claim that under the bounds \eqref{bmd-1} and \eqref{bmd-2} and given $R$, the event 
\[
\U^{(k+1)}=R \wedge \G^{(k+1)}=a \wedge \S_{k+1}=N \wedge E_{k+1} \wedge \G_{1,k+1}=M
\]
is non-empty for at most one pair $(a,N)\in \N^{k+1}\times \Z/q^n\Z$. We leave the proof of this claim for the next step. We deduce that the left hand side of \eqref{half} is bounded by 
\begin{align*}
& \ls_{A'} n^{-2A'} q^n p^{-M} \ll n^{-2A'}n^{ C_A^2\log p} \ll n^{-4A}
\end{align*}
 if $A'\gtrsim_A 1$. This finishes the proof.
 \medskip
 
\noindent {\bf Step 2}. We now prove the remaining claim. First we note that for each $R\in \{r(1),\ldots,r(p-1)\}^l$, the map $F_l(\cdot,R)\colon \N^l \to \Z[\frac{1}{p}]$ is injective.
Indeed,  assume that $F_n(a,R) = F_n(a',R)$.  Using formula \eqref{eq: Fk formula} and taking $p$-valuations on both sides we obtain that
$$ a_{1,l} = a'_{1,l}.$$
Moreover, by formula \eqref{eq:Fnrecursionformula} we have
$$ F_l(a,R) = q^{l-1} p^{-a_{1,l}}R_1 + F_{l-1}(a^{(2,l)},R^{(2,l)})$$
and similarly for $a'$. Therefore
$$ F_{l-1}(a^{(2,l)},R^{(2,l)}) = F_{l-1}(a'^{(2,l)},R^{(2,l)}).$$
The results follows by induction. Below we present a $q$-adic version.

\begin{lem}
Let $k\leq n$, $M$ obey \eqref{bmd-2} and $R\in \{r(1),\ldots,r(p-1)\}^{k+1}$. Then the residue classes $F_{k+1}(a,R) \mod  q^n$ as $a\in \N^{k+1}$ varies over tuples obeying the two conditions
\begin{equation}\label{eq: q-adic res 2}
a_{1,k+1} = M
\end{equation}
and
\begin{equation}\label{eq: q-adic res}
 |a_{i,j} - \mu(j-i+1)| \leq C_A \left( \sqrt{(j-i+1)(\log n)} + \log n \right)
\end{equation}
for $1 \leq i \leq j \leq k+1$,  are all distinct (if $C_A$ is large).
\end{lem}

\begin{proof}  Suppose that $a, a' \in \N^l$ are two tuples that both obey  \eqref{eq: q-adic res 2} and \eqref{eq: q-adic res}, and such that
$F_{k+1}(a,R) \mod q^{n} = F_{k+1}(a',R) \mod q^{n}$.
Using \eqref{eq: Fk formula} and multiplying by $p^M$ we conclude that
\begin{equation}\label{lrs}
 \sum_{i=1}^{{k+1}} q^{i-1} p^{M - a_{k+2-i,k+1}}R_{k+2-i}  \equiv  \sum_{i=1}^{{k+1}} q^{i-1} p^{M - a'_{k+2-i,k+1}}R_{k+2-i} \ \m{ q^{n}}.
\end{equation}
 Using \eqref{eq: q-adic res} and \eqref{bmd-2} we have
\begin{align*}
\left| \sum_{i=1}^{{k+1}} q^{i-1} p^{M - a_{k+2-i,k+1}}R_{k+2-i}\right| & \ls p^M \sum_{i=1}^{{k+1}} q^{i} p^{- \mu i + C_A (\sqrt{i \log n} + \log n)} \\
& \leq  n^{(C_A- \tfrac12 C_A^3) \log p } q^{n} \sum_{i=1}^{{k+1}} p^{-(\mu-\log_p q)i + C_A \sqrt{i \log n}} \\
& \leq  n^{(C_A + \tfrac1c C_A^2- \tfrac12 C_A^3) \log p } \sum_{i=1}^{{k+1}} p^{-(\mu-\log_p q-c)i } \\
& \ll_{C_A}  n^{ - \tfrac14 C_A^3 \log p } q^{n}.
\end{align*}
where above we have used the inequality $|xy| \leq cx^2+\tfrac1c y^2$.  In particular, for $n$ large enough, this expression is less than $\tfrac12 q^n$.  Similarly for the right-hand side of \eqref{lrs}.  Thus the two sides of \eqref{lrs} are equal as natural numbers, not simply as residue classes modulo $ q^n$. We conclude $F_{{k+1}}(a,R) = F_{{k+1}}(a',R)$.  The claim follows from the injectivity of $F_{k+1}(\cdot ,R)$.
\end{proof}
We then see that, given $M$ obeying \eqref{bmd-2} and $R\in \{r(1),\ldots,r(p-1)\}^{k+1}$, the event
\[
\U^{(k+1)}=R \wedge \G^{(k+1)}=a \wedge \S_{k+1}=N \wedge E_{k+1} \wedge \G_{1,k+1}=M
\]
implies that $F_{k+1}(a,R)=N$, which by the previous lemma  and the injectivity of $F_{k+1}(\cdot ,R)$ uniquely defines the pair $(a,N) \in \N^k \times \Z/q^n\Z$, which proves the desired claim.
\end{proof}



\section{Proof of the Main Result}\label{sec:proofmain}
In this section, by using our previous results, we conclude \thmref{thm:main}. It is now enough to prove Proposition \ref{prop:decaycharacteristic} and we do this assuming Theorem \ref{thm:renewal}.

\begin{proof}[{\bf Proof that Theorem \ref{thm:renewal} $\Rightarrow$ Proposition \ref{prop:decaycharacteristic}}]
$\nonumber$

\noindent {\bf Step 1.} First we use \eqref{eq:syracRVproj} to group neighbour terms to get
$$
\S_n = \sum_{j=1}^{n/2} q^{2j-2}p^{-\G_{1,2j}}(p^{\G_{2j}}\U_{2j-1} + q \U_{2j}) \mod q^{n} \equiv \sum_{j=1}^{n/2} q^{2j-2}p^{-\P_{1,j}}(p^{\G_{2j}}\U_{2j-1} + q \U_{2j}) \mod q^{n}
$$
when $n$ is even, where $\P=(\P_1,\P_2,\ldots)$ is a vector with i.i.d. copies of the Pascal distribution $\P_j \equiv \G_1+\G_2$, in particular
$$
\p{\P_j = N} = (p-1)^2 (N-1)p^{-N}.
$$
A similar formula holds when $n$ is odd with an extra term in the end. Let $\xi \in \Z/q^n\Z$ not divisible by $q^n$ and define
$$
\chi(x):= e^{-2\pi i \xi (x \mod q^n) /q^n}.
$$
so $\chi(\S)=\chi(\S_n)=e^{-2\pi i \xi \S_n /q^n}$.
Fubini's theorem and independence allows to bound
\begin{align}\label{ineq:expectationfbound}
|\E[\chi(\S)]|  \leq \e{\prod_{j=1}^{\hn-B}|f(q^{2j-2} p^{-\P_{1,j}},\P_j)|}
\end{align}
where $B=B(p,q,r)\in \N$ is to be chosen later on and where
$$
f(x,N)=\e{\chi(x(p^{\G_1}\U_1+q\U_2)) | \G_1+\G_2=N}.
$$
The same bound holds for odd $n$ since the extra term can be bounded by $1$.
Since $q$ does not divide $\xi$ and by assumption $\gcd(q,r(1),\ldots,r(p-1))=1$, there must be some $j_0$ such that $q$ does not divide $\xi r(j_0)$. We then define $\theta(j,l) \in (-1/2,1/2]$  by 
\begin{align}\label{def:theta}
\theta(j,l):=\left\{\frac{\xi r(j_0) (p-1)q^{2j-2}(p^{1-l} \mod q^n)}{q^n}\right\}_1 \quad {\rm for}\  (j,l)\in \N^2,
\end{align}
where we let $\{x\}_{1}$ is the unique real in the set $(x+\Z)\cap (-1/2,1/2]$.

For a given $\ep>0$ we now let 
\[
\B = \{(j,l)\in \N^2 : |\theta(j,l)|< \ep\}.
\]
We claim that $\B$ \emph{is} a disjoint union of triangles of the form \eqref{def:triangles} with slope $\eta=\log_p q^2$, and that those which are in contact with the region $\{j\leq \hn-B, l\geq 1\}$ are separated from each other and from the edge $\{\hn-B\}\times \N$ by $c \log(1/\ep)$. We postpone the proof of this claim to Step {$2$}. We have that 
\begin{align*}
|f(q^{2j-2} p^{-l},3)| & = \frac{1}{(p-1)^2}\left|\sum_{i,j=1}^{p-1} \frac12 \chi(x(p r(i) + qr(j))) + \frac12 \chi(x(p^2 r(i) + qr(j)))\right|_{x=q^{2j-2} p^{-l}} \\
& \leq \frac{1}{p-1}\sum_{i=1}^{p-1} \left| \frac{1 +  \chi(xp(p-1) r(i))}{2}\right|_{x=q^{2j-2} p^{-l}} \\ 
& = \frac{1}{p-1}\sum_{i=1}^{p-1} |\cos(\pi \xi q^{2j-2} p^{1-l}(p-1)r(i)/q^n)| \\
& \leq1-\frac{1-|\cos(\pi \theta(j,l))|}{p-1},
\end{align*}
where above we bounded all terms by $1$ except the term with $i=j_0$, and have used that by construction $|\theta(j,l)|\leq \tfrac{1}{2}$. Thus, if $(j,l)\notin \B$ then $\ep< |\theta(j,l)|\leq \tfrac12$ and we have 
\[
|f(q^{2j-2} p^{-l},3)|\leq 1-c\ep^2 \leq e^{-\ep^3}
\]
for small $\ep>0$. Therefore, a simple way to bound the right hand side of \eqref{ineq:expectationfbound} is
\[
\e{\prod_{j=1}^{\hn-B}|f(q^{2j-2} p^{-\P_{1,j}},\P_j)|}  \leq \e{e^{-\ep^3 \#\{  j\in[n/2-B]\,:\, \P_{j}=3, \, \P_{1,j}=l, \, (j,l)\notin \B\}}}.
\]
Now note that if we define $\H:=(\J,\P_{1,\J})$, where $\J$ is the first time such that $\P_{\J}=3$ (so $\P_l\neq 3$ for $l<\J$), then we have
\[
\#\{ j\in[n/2-B]\,:\, \P_{j}=3, \,(j, \P_{1,j}) \notin \B\} \equiv \#\{ k \in [n/2-B] \,:\, \H_{1,k}\notin \B \} 
\]
where $(\H_1,\H_2,\ldots)$ are i.i.d. copies of $\H$ and $[n/2-B]=\{1,2,\ldots,\hn-B\}$. Hence,
\begin{align}\label{ineq:badset}
|\E[e^{-2\pi i \xi \S_n/q^n}]| \leq \e{e^{-\ep^3\#\{ k \in [n/2-B] \,:\, \H_{1,k}\notin \B \}}}.
\end{align}
A straightforward computation shows that $\J\equiv \G(\la)$ with $\la=\tfrac{p^3}{2(p-1)^2}$ and that 
\[
\E[\H]=\la(1,2\mu).
\]
(one could guess this formula by the heuristic $\H \approx (\la,\sum_{k=1}^\la 2\mu)=\la(1,2\mu)$). In particular, $\H$ has expected slope $2\mu >  \log_p q^2=\eta$. It is also not hard to show that
\[
\infty>\e{e^{2c|\H|}}  = \int_0^\infty \p{|\H|\geq t}e^{2ct}dt
\]
if $c$ is sufficiently small. We conclude that $\H$ has exponential tail
$
\p{|\H|\geq t} \ls e^{-c t}.
$
Clearly $\H$ is not supported in a coset of a proper subgroup of $\Z^2$ since any point $\N \times (1+\N)$ has positive probability (alternatively, one can compute directly the covariance matrix of $\H$ and see that it is positive definite). We have now satisfied all conditions of Theorem \ref{thm:renewal} for the renewal process $\X_{1,k}=\H_{1,k}$, with $\omega=2\mu$, $\ep$ replaced by $\ep^3$ and $n$ by $\hn-B$. We can then apply Theorem \ref{thm:renewal} in conjunction with \eqref{ineq:expectationfbound} and \eqref{ineq:badset} to conclude that for any $A>0$ we have
\[
|\E[e^{-2\pi i \xi \S_n/q^n}]| \ls_{A,\ep} n^{-A}.
\]
\medskip

\noindent {\bf Step 2.} It now remains to prove the claim about $\B$. In what follows $\ep$ is to be taken sufficiently small (depending only on $p,q,r,\xi$). {First we show that if $(j,l)\in \B$ and $j\leq \hn -B$ then $j\leq \hn -B-c\log(1/\ep)$. Indeed, if $|\theta(j,l)|<\ep$ and $j\leq \hn -B$ then we can multiply \eqref{def:theta} by $q^{n-2j-2B}$ to obtain
\begin{align}\label{id:separationj=n/2}
\frac{\xi r(j_0)(p-1)(p^{1-l} \mod q^n)}{q^{2B+2}} \in q^{n-2j-2B}(-\ep,\ep) + \Z
\end{align}
Since $q$ does not divide $\xi r(j_0)$, there is a prime $a$ such that $\xi r(j_0)=a^s u$, $q=a^{s'} u'$ with $s'>s$ and $a$ does not divide $u$ and $u'$. Let $p-1=a^{s''} u''$ where $a$ does not divide $u''$. Since $(p^{1-l} \mod q^n)$ is coprime with $q$ we deduce that $\xi r(j_0)(p-1)(p^{1-l} \mod q^n)=a^{s+s''}u'''$ with $u'''$ coprime with $a$. In particular, if we take $B$ such that $(2B+2)s'=s+s''+1$ we conclude that $\{\xi r(j_0)(p-1)(p^{1-l} \mod q^n))/q^{2B+2}\}_1 =k/q^{2B+2}$ for some $k\neq 0$. Hence $q^{n-2j-2B}\ep \geq 1/q^{2B+2}$ and so $j\leq \hn +1 - B  - \tfrac12\log(q^{-2B-2}/\ep)\leq \hn - B-c\log (1/\ep)$. }

Now observe that by \eqref{def:theta} we have
\begin{align*}
\theta(j',l')& = \{p^{\eta(j'-j)+(l-l')} \theta (j,l) \}_1 \quad (j'\geq j \text{ and } l'\leq l)\\
\theta(j,l) & =\{u\theta(j+1,l) + v \theta(j,l-1)\}_1,
\end{align*}
where we set $\eta=\log_p q^2$ and $u,v\in \Z$ are uniquely chosen such that $1=uq^2+vp$.  Note first that  if $(j_0,l_0)\in \B$, $|\theta(j_0,l_0)|=\ep p^{-s_0}$ for some $s_0>0$ and
$$
\D(j_0,l_0,s_0) := \{(j,l)\in \N^2 : j \geq j_0, l\leq l_0, (j -j_0)\eta + (l_0-l) \leq s_0\},
$$
then it is easy to see that $\D(j_0,l_0,s_0) \subset \B$. Indeed, if $(j,l)\in \D(j_0,l_0,s_0)$ then 
\begin{align*}
\log_p |\theta(j,l)| & = \log_p | \{q^{2(j-j_0)}p^{l_0-l}\theta(j_0,l_0)\}_1 |  \leq (j-j_0)\eta + (l_0-l)  + \log_p \ep - s_0 \leq \log_p \ep.
 \end{align*}
 Now let $(j_1,l_1)\in \B$ be given and then, moving inside $\B$, go up as far as possible and then left as far as possible, arriving at a point $(j_0,l_0) \in \B$, with $|\theta(j_0,l_0)|=\ep p^{-s_0}$. Let $\D^\ep$ be a $c \log(1/\ep)$-neighbourhood of $\D:=\D(j_0,l_0,s_0)$. It is now enough to show that \[
 (\D^\ep \setminus \D) \cap \B = \emptyset.
 \]
Indeed, this guarantees that $(j_1,l_1)\in \D$ and that triangles are $c\log(1/\ep)$-separated. Let $(j,l)\in \D^\ep \setminus \D$ and $(j_2,l_2)\in \D$ be such that 
\[
|j_2-j|+|l_2-l|={\rm dist}((j,l),\D) < c \log(1/\ep).
\]
Clearly, one of the following cases must hold: 
\begin{enumerate}
\item[(A)] $j<j_0$ and then $(j_2,l_2)=(j_0,l)$; 
\item[(B)] $j\geq j_0$, $l>l_0$ and then $(j_2,l_2)=(j,l_0)$; 
\item[(C)] $j\geq j_0$, $l\leq l_0$, $s_0 < \eta (j-j_0) + (l_0-l) $ and then $s_0 \ls (j_2-j_0)\eta+(l_0-l_2) \leq s_0$.
\end{enumerate} 
In cases (A), (B) and (C) we respectively have
\begin{align}
\eta(j-j_0) + (l_0-l) &=\eta(j-j_2) + (l_0-l_2) < c \eta \log(1/\ep) + s_0,\\
\eta(j-j_0) + (l_0-l)&=\eta(j_2-j_0) + (l_2-l) < s_0+ c \log(1/\ep),\\
\eta(j-j_0) + (l_0-l) & = \eta(j-j_2) + (l_2-l) + \eta(j_2-j_0) + (l_0-l_2)< \eta c \log(1/\ep) +s_0,
\end{align}
that is, in any case we have 
\begin{align}\label{ineq:cases}
\eta(j-j_0) + (l_0-l) < s_0 + c \eta\log(1/\ep).
\end{align}
In particular
\[
|p^{\eta(j-j_0)+(l_0-l)}\theta(j_0,l_0)| \leq \ep^{1-c\eta},
\]
and since $\theta(j,l)=\{p^{\eta(j-j_0)+(l_0-l)}\theta(j_0,l_0)\}_1$ we deduce that 
\[
\theta(j,l)=p^{\eta(j-j_0)+(l_0-l)}\theta(j_0,l_0).
\] 
In case $(C)$ this already shows that $|\theta(j,l)|>\ep$ and so $(j,l)\notin \B$.  In case (B) the idea is to use that by construction $(j_1,l_0+1)\notin \B$ and $(j_1,l_0)\in\B$, and thus since $\theta(j_1,l_0)=\{ p\theta(j_1,l_0+1)\}_1$, we must have $\theta(j_1,l_0)\neq p\theta(j_1,l_0+1)$.  We proceed with the proof. In this case we have $l-l_0< c\log(1/\ep)$ and by \eqref{ineq:cases} we have  $\eta(j-j_0) < c(\eta+1)\log(1/\ep)$. Assume by contradiction that $(j,l)\in \B$, that is, $|\theta(j,l)|<\ep$. We obtain
\[
|p^{l-l_0-1}\theta(j,l)| < p^{-1}\ep^{1-c},
\]
and since $\theta(j,l_0+1)=\{p^{l-l_0-1}\theta(j,l)\}_1$ we deduce $\theta(j,l_0+1)=p^{l-l_0-1}\theta(j,l)$ and so $p|\theta(j,l_0+1)|<\ep^{1-c}$. Hence $\theta(j,l_0)=p\theta(j,l_0+1)$ and $|\theta(j,l_0)|<\ep^{1-c}$. In particular, by construction, $j_1\neq j$. Assume now $j_1<j$. Assume by induction that $\theta(j-k,l_0)=p\theta(j-k,l_0+1)$ and $|\theta(j-k,l_0)|<\ep^{1-(\eta+1)c}$ for some $k
\geq 0$ such that $j_1<j-k$. First notice that since $|p^{\eta(j-k-j_0)}\theta(j_0,l_0)|< \ep^{1-c(\eta+1)}$, then $\theta(j-k,l_0)=p^{\eta(j-k-j_0)}\theta(j_0,l_0)$ and
$|\theta(j-k,l_0+1)| = |\theta(j-k,l_0)/p| < p^{-1}\ep^{1-c(\eta+1)}$. Secondly, that since $\theta(j-k-1,l_0)=\{p^{\eta(j-k-1-j_0)}\theta(j_0,l_0)\}_1$ and $|p^{\eta(j-k-1-j_0)}\theta(j_0,l_0)|< \ep^{1-c(\eta+1)}$, we have $\theta(j-k-1,l_0)=p^{\eta(j-k-1-j_0)}\theta(j_0,l_0)$ and so $|\theta(j-k-1,l_0)|< \ep^{1-c(\eta+1)}$. These two inequalities in conjunction with
\[
\theta(j-k-1,l_0+1)=\{u\theta(j-k,l_0+1)+v\theta(j-k-1,l_0)\}_1
\]
show that $\theta(j-k-1,l_0+1)=u\theta(j-k,l_0+1)+v\theta(j-k-1,l_0)$ and $|\theta(j-k-1,l_0+1)|< 2(|u|+|v|) \ep^{1-c(\eta+1)}$. We conclude that $p\theta(j-k-1,l_0+1)=\theta(j-k-1,l_0)$ and $|\theta(j-k-1,l_0)|<\ep^{1-c(\eta+1)}$. We deduce by induction that $p\theta(j_1,l_0+1)=\theta(j_1,l_0)$, which is absurd. Assume now $j_1>j$. Assume again by induction that $\theta(j+k,l_0)=p\theta(j+k,l_0+1)$ for some $k\geq 0$ such that $j+k<j_1$. Since $\theta(j+k+1,l_0+1)=\{p^{\eta}\theta(j+k,l_0+1)\}_1$ and $|p^{\eta}\theta(j+k,l_0+1)| =p^{\eta-1}|\theta(j+k,l_0)|<p^{\eta-1}\ep$ (because $(j+k,l_0)\in \B$ by construction) we deduce $\theta(j+k+1,l_0+1)=p^{\eta}\theta(j+k,l_0+1)$ and so $p|\theta(j+k+1,l_0+1)|<p^{\eta}\ep$. Since $\theta(j+k+1,l_0)=\{p\theta(j+k+1,l_0+1)\}_1$ we obtain $\theta(j+k+1,l_0)=p\theta(j+k+1,l_0+1)$. We conclude by induction that $\theta(j_1,l_0)=p\theta(j_1,l_0+1)$, which is absurd. In case (A) the argument is similar and uses the fact that $(j_0,l_0)\in \B$ but $(j_0,l_0-1)\notin \B$. We leave the details to the reader.
\end{proof}

\section{Proof of Theorem \ref{thm:renewal}}\label{sec:proofrenewal}
As usual $c$ will represent arbitrarily small positive numerical quantity that can change from line to line.

\noindent {\bf Step 1.} Without loss of generality we can assume $\B=\sqcup \D(j_\D,l_\D,s_\D)$. By classical Chernoff-type bounds (see \cite[Lemma 2.2]{T}) we have
\begin{align}\label{ineq:decayX}
\begin{split}
 \p{\X_{1,k}=(j,l)} &\ls k^{-1}G_k(c((j,l)-k\e{\X})) 
 \\  \p{|\X_{1,k}-k\e{\X}|>t}  &\ls G_k(c t),
 \end{split}
\end{align}
where $G_k(y)=e^{-|y|} + e^{-y\cdot y /k}$. For a given $s\in \N$ we define $\K_s:=\min\{k\in\N: \L_{1,k}>s\}$. We claim that 
\begin{align}\label{ineq:decayfirsttime}
\p{\X_{1,\K_s}=(j,l)} \ls s^{-1/2} e^{-c(l-s)}G_{s+1}(c(\omega j-s)),
\end{align}
which easily implies that 
\begin{align}\label{ineq:decayfirsttime2}
\p{\J_{1,\K_s}=j}  \ls s^{-1/2} G_{s+1}(c(\omega j-s)) \quad \text{and} \quad \p{\L_{1,\K_s}=l}  \ls \1_{l>s}e^{-c(l-s)}.
\end{align}
Note \eqref{ineq:decayfirsttime} is essentially saying that $\X_{1,\K_s}$ is roughly uniformly distributed in the set $j=s/\omega + O(\sqrt{s})$ and $s<l<s + O(1)$, and indeed one can show that if we let
\[
U_B:=\{(j,l)\in \N^2 \, :\,  |j-s/\omega| \leq B s^{1/2}, \ s<l<s+B\}.
\]
Then we have
\begin{align}\label{ineq:lowerboundunif}
\p{\X_{1,\K_s}\in U_B} \gtrsim 1 \quad \text{for} \quad B\gtrsim 1.
\end{align}
To show \eqref{ineq:decayfirsttime} first note that, by construction, if $\K_s=k$ and $l=\L_{1,k}>s\geq \L_{1,k-1}=l-\L_k$ then $\L_k\geq l-s$ and we obtain
\begin{align*}
\p{\X_{1,\K_s}=(j,l)} &\leq \sum_{k\geq 1} \p{\X_{1,k}=(j,l) \wedge \L_k\geq l-s} \\
& = \sum_{k\geq 1} \sum_{\substack{j' < j \\ l' \geq l-s \geq 1}}\p{\X_{1,k-1}=(j-j',l-l')}\p{\X_k=(j',l')}\\
& \ls  \sum_{k\geq 1} \sum_{\substack{ 1\leq j' < j \\ l' \geq l-s \geq 1}} k^{-1}G_{k}(c^2(j-j',l-l')-c^2(k-1)\e{\X})e^{-c(l'+j')} \\
& = e^{-c(l-s)}\sum_{k\geq 1} \sum_{\substack{ 1\leq j' < j \\ 1\leq l'' \leq s}} k^{-1}G_{k}(c^2(j-j',s-l'')-c^2(k-1)\e{\X})e^{-c(l''+j')} \\
& \ls e^{-c(l-s)}\sum_{k\geq 1} k^{-1}G_{k}(c(j,s)-c(k-1)\e{\X}) \\
& \ls s^{-1/2}e^{-c(l-s)}G_{s+1}(c(\omega j-s)),
\end{align*}
where in the third inequality we have used the exponential tail decay of $\X$ and the bound \eqref{ineq:decayX}. The last two inequalities are routine computations.

We now let $\overline{\B}=\N^2 \setminus \B$ and define
\begin{align}\label{def:Q}
Q(j,l):=\e{ \exp(-\ep \#\{k \geq 0  :  (j,l) + \X_{1,k} \notin \B\})},
\end{align}
(with $\X_{1,0}=0$) for all $(j,l)\in \N^2$. Notice that
\begin{align}\label{id:Qrecursion}
Q(j,l)=e^{-\ep \1_{\overline{\B}}(j,l)}\e{Q((j,l)+\X)}
\end{align}
and
\[
\e{\exp\left({-\ep \#\{k\geq 1 : \X_{1,k} \notin \B\}}\right)}=\e{Q(\X)}.
\]
Our goal now is to show $\e{Q(\X)} \ls_{A,\ep} n^{-A}$ for any given $A\in \N$. For $m\geq 1$ we let 
$$
Q_m := \max_{\substack{j\geq n-m \\ l\geq 1}}(n-j)_+^A Q(j,l)
$$
where $(x)_+=\max(1,x)$. We claim that there is $C_{A,\ep}>0$ such that if $Q_{m} \geq  C_{A,\ep}$ for some $m$ then $Q_{m'}=Q_{m}$ for all $m'\geq m$. In other words, the quantity $Q_m$ is eventually constant and it plateaus at some point $m_{A,\ep}$. We postpone the proof of this claim to the next steps. Assuming the claim is true the theorem follows easily since
\[
Q(\X) \ls_{A,\ep} (n-\J)^{-A} \leq 2^A n^{-A}\J^A
\]
and so $\e{Q(\X)} \ls_{\ep,A} n^{-A} \e{\J^A} \ls_A  n^{-A}$, since $\e{\J^A}<\infty$.
\medskip

\noindent {\bf Step 2.}
In the following steps we stablish the remaining claim about the quantity $Q_m$ from Step 1. Since $Q_m \leq Q_{m+1}$ for all $m\geq 1$ it not hard to realize that it is sufficient to show that
\begin{align}\label{ineq:goalforQ}
Q(n-m,l) \leq m^{-A}Q_{m-1} \quad \text{for } m\geq m_{A,\ep} \, \text{ and } \, l\geq 1.
\end{align}
The idea now is to exploit the recursion formula \eqref{id:Qrecursion} and separate the analysis into the following three cases: 
\begin{enumerate}
\item[(I)] $(n-m,l) \notin \B$;
\item[(II)] $(n-m,l) \in \D(j_\D,l_\D,s_\D)$  and $l \geq \l_\D - m/\log^2(m)$;
\item[(III)] $(n-m,l) \in \D(j_\D,l_\D,s_\D)$ and $l \leq l_\D- m/\log^2(m)$. 
\end{enumerate}

\noindent \emph{Case} (I). Here we can directly use \eqref{id:Qrecursion} and the fact that $\J\geq 1$ a.s. to obtain
\[
Q(n-m,l) \leq e^{-\ep} Q_{m-1}\e{(m-\J)_+^{-A}} \leq e^{-\ep} m^{-A}Q_{m-1}\e{e^{2A\J\frac{\log m}{m}}},
\]
where above we have used the inequality 
\begin{align}\label{ineq:weirdineq}
(m-j)^{-1}_+ \leq m^{-1}e^{2j\frac{\log m}{m}}.
\end{align}
Since $\E[e^{2A\J\frac{\log m}{m}}]\to 1$ as $m\to\infty$, there must be $m_{A,\ep}$ such that $\E[e^{2A\J\frac{\log m}{m}}] \leq e^{\ep}$ for $m\geq m_{A,\ep}$ and we obtain \eqref{ineq:goalforQ}.
\medskip

\noindent \emph{Case} (II). We let $s=l_\D-l=O(\frac{m}{\log^2 m})$ and  use the definition \eqref{def:Q} to obtain
\begin{align*}
Q(n-m,l) & = \e{\exp({-\ep \#\{0\leq k<\K_s : (n-m,l)+\X_{1,k} \notin \B\}})Q((n-m,l)+\X_{1,\K_s})} \\
&\leq \e{Q((n-m,l)+\X_{1,\K_s})} \\
& \leq Q_{m-1}\e{e^{-\ep\1_{\overline{\B}}((n-m,l)+\X_{1,\K_s})}(m-\J_{1,\K_s})_+^{-A}} \\
& \leq m^{-A}Q_{m-1}\e{\exp\left(2A\J_{1,\K_s}\tfrac{\log m}{m}-\ep\1_{\overline{\B}}((n-m,l)+\X_{1,\K_s}\right)} \\
& \leq m^{-A}Q_{m-1}\big(\e{\exp\left(2A\J_{1,\K_s}\tfrac{\log m}{m}\right)}-\ep^2\p{(n-m,l)+\X_{1,\K_s}\notin \B}\big),
\end{align*}
where in the second inequality we  have used the recursion \eqref{id:Qrecursion} and the definition of $Q_{m-1}$, in the third the bound \eqref{ineq:weirdineq} and in the fourth that $e^{-\ep} \leq 1-\ep^2$ for small $\ep$.  It is not hard to show using \eqref{ineq:decayfirsttime2} that the function $\theta\mapsto \E[\exp( \frac{\theta}{s}\J_{1,\K_s})]$ is smooth for small $|\theta|\leq c \omega$ and so
\[
\e{\exp\left(2A\J_{1,\K_s}\tfrac{\log m}{m}\right)} = 1+O(As \tfrac{\log m}{m}) = 1+O(\tfrac{A}{\log m}).
\]
Since $(n-m,l)\in \D(j_\D,l_\D,s_\D)$ and $s=l_\D-l$ we have $\eta(n-m-j_\D)+s \leq s_\D$. Therefore, if $(n-m,l)+\X_{1,\K_s}=(j',l')$ with $l'=l_\D+O(1)$ and $j'=n-m+s/\omega+O(s^{1/2})$ we have
\[
\eta(j'-j_\D) \leq s_\D-s(1-\eta/\omega) +O(s^{1/2}) \leq s_\D+O(1),
\]
where we have used the hypothesis that $\omega>\eta$. Thus $(j',l')\in (\D(j_\D,l_\D,s_\D) + O(1))\setminus \D(j_\D,l_\D,s_\D)$. Due to the separation $\gtrsim \log (1/\ep)$ of the triangles, one can choose $\ep$ sufficiently small so to guarantee that $(j',l')\notin \B$. Using \eqref{ineq:lowerboundunif} we can find $B$ sufficiently large such that $\X_{1,\K_s} \in U_B \Rightarrow (n-m,l)+\X_{1,\K_s}\notin \B$ and we obtain
\begin{align}\label{ineq:lowerboundnotinB}
\p{(n-m,l)+\X_{1,\K_s}\notin \B} \geq \p{\X_{1,\K_s} \in U_B} \geq  c.
\end{align}
We conclude that $Q(n-m,l) \leq (1+O(A/\log m)-c\ep^2)m^{-A}Q_{m-1}$ and thus \eqref{ineq:goalforQ} follows.
\medskip

\noindent \emph{Case} (III). Roughly speaking, if $(n-m,l)$ is deep inside a triangle then one needs to find several steps outside $\B$ to account for the degradation suffered. Again we select $s=l_\D-l$, and so $s\geq \tfrac{m}{\log^2 m}$. Also note that since $\eta(n-m-j_\D)+s \leq s_\D$ and, by assumption, triangles are for away from the edge $\{n\}\times \N$, that is, $j_\D+s_\D/\eta \leq n-c \log(1/\ep)$, then we conclude 
\begin{align}\label{ineq:sbounds}
\frac{m}{\log^2 m}\leq s\leq \eta m.
\end{align}
As before we start with the inequality
\begin{align*}
&Q(n-m,l) \\ &\leq \e{Q((n-m,l)+\X_{1,\K_s})} \\
& =\e{\exp({-\ep \#\{0\leq k<\K_s + P: (n-m,l)+\X_{1,k} \notin \B\}})Q((n-m,l)+\X_{1,\K_s+P})}\\
& \leq Q_{m-1}\e{\exp\left({-\ep\sum_{k=0}^P \1_{\overline{\B}}((n-m,l)+\X_{1,\K_s+P})}\right)(m-\J_{1,\K_s+P})_+^{-A}},
\end{align*}
where $P=P_{A,\ep}$ is to be chosen later. It is  now suffices to show that
\begin{align}\label{ineq:Qgoal2}
\e{\exp\left({-\ep\sum_{k=0}^P \1_{\overline{\B}}((n-m,l)+\X_{1,\K_s+k})}\right)(m-\J_{1,\K_s+P})_+^{-A}} \leq m^{-A}
\end{align}
for $m\geq m_{A,\ep}$.
Let $\ga=\tfrac12 + \tfrac{\eta}{2\omega}$ so that 
\[
\frac{\eta}{\omega}<\ga<1.
\]
We have
\begin{align}\label{ineq:tailofJKs}
\p{\J_{1,\K_s+P} \geq \ga m} \leq \p{\J_{1,\K_s} \geq (\ga-c)m} + \p{\J_{\K_s+1,\K_s+P} \geq c m} \ls_P e^{-cm},
\end{align}
where for the first term in the second inequality we have used \eqref{ineq:decayfirsttime2} with the upper bound of $s$ in \eqref{ineq:sbounds}, and in the  second term we use the tail bound \eqref{ineq:jansontail} for $\J\equiv \G(\la)$. Since the conditional expectation under the event $\J_{1,\K_s+P} \geq \ga m$ of the left hand side in \eqref{ineq:Qgoal2} is $O(m^{A}e^{-cm})$, we can condition the left hand side of \eqref{ineq:Qgoal2} to the event $\J_{1,\K_s+P} < \ga m$. However, since in this situation we have $(m-\J_{1,\K_s+P})_+^{-A} \leq (1-\ga)^{-A}m^{-A}$ we conclude that \eqref{ineq:Qgoal2} will follow if we show
\[
\e{\exp\left({-\ep\sum_{k=0}^P \1_{\overline{\B}}((n-m,l)+\X_{1,\K_s+k})}\right)} \leq (1-\ga)^{A+1}
\]
for large $m$, which in turn is implied (by conditional expectation\footnote{We are using the identity $\E[\Y]=\E[\Y|E]\p{E}+\E[Y|E']\p{E'}$ when $\Y$ is a random variable and $E$ is an event with complement $E'$.}) by
\begin{align}\label{ineq:Qgoal3}
\p{\sum_{k=1}^P \1_{\overline{\B}}((n-m,l)+\X_{1,\K_s+k}) \leq \frac{A \lambda}{\ep}} \leq \ga(1-\ga)^{A+1}
\end{align}
for $m\geq m_{A,\ep}$ (and small $\ep$), where $B=4\log(1/(1-\ga))$.
\medskip

\noindent {\bf Step 3.}
We now state two lemmas and we postpone their proofs for the next steps.

\begin{lem}\label{lem:largetriangles}
Assume case (III). Let $k,s'\in \N$ and define the event
\[
E_{k,s'}:= (n-m,l)+\X_{1,\K_{s}+k} \in \D' \,\wedge\, s_{\D'}\geq s',
\]
for some $\D'\subset \B$, where $s=l_\D-l\geq \tfrac{m}{\log^2 m}$. Then if $1\leq s' \leq m^{1/3}$  and $m\geq m_{A}$ we have\footnote{The exponent $1/3$ could be replaced by anything less than $1/2$.}
\[
\p{E_{k,s'}} \ls A^2 \frac{k}{s'} + e^{-cA^2k}.
\]
\end{lem}

We now recursively define a sequence of stopping times
\[
(\T_1(j,l),\T_2(j,l),\ldots, \T_\I(j,l))
\]
for a given $(j,l)\in \N^2$ such that
\[
\T_i(j,l)=\min\{k\in \N: (j,l)+\X_{1,k}(j,l) \in \B \,\wedge\, l+\L_{1,k}>l_{\D_{i-1}}\}
\] 
and select a triangle $\D_i=\D(j_{\D_i},l_{\D_i},s_{\D_i})\subset \B$ uniquely such that $(j,l)+\X_{1,i}(j,l)\in \D_i$ (with the convention that $l_{\D_0}=0$). In other words, the stopping times $\T_{i}(j,l)$ are the first instants where the random walk $(j,l)+\X_{1,k}$ enters a new triangle of $\B$ that lies strictly above the previous triangle. Above $\I=\I(j,l)$ is the total number of stopping times that can be constructed in this way. If $\T_1(j,l)$ does not exist then we set $\I(j,l)=0$ (and conventionally $\T_0 = 0$).  We claim that if the random walk $(j,l)+\X_{1,k}$ spends a lot of time in $\B$, then it must also spend some significant time outside $\B$.

\begin{lem}\label{lem:manytriangles}
If $\ep$ is sufficiently small we have (omitting the variables $(j,l)$)
\begin{align}\label{ineq:manytriangles}
\e{\exp\left(\ep \I_a-\sum_{k=1}^{\T_{\I_a}}\1_{\overline{\B}}((j,l)+\X_{1,k} ) \right)} \leq e^\ep
\end{align}
for all $(j,l)\in\N^2$ and  $a\in \N$, where $\I_a=\min(a,\I)$.
\end{lem}

We now finish the proof of Theorem \ref{thm:renewal} by showing \eqref{ineq:Qgoal3}. We first apply Lemma \ref{lem:largetriangles} to obtain that the event\footnote{The exponent $1/7$ could be anything smaller than $1/3$.} 
\[
E = \bigvee_{1\leq k \leq m^{1/7}} E_{k, \lceil k^3(1-\ga)^{-2A} \rceil}
\]
has small probability $\p{E}\ls A^2 (1-\ga)^{2A}$.  Without loss of generality we can assume $A\geq A_{\ga}$ to obtain 
\[
\p{E} \leq \tfrac{\ga}{2}(1-\ga)^{A+1}.
\]
Similarly define the event 
\[
F:=\left[\exp\left(\ep \I_a-\sum_{k=1}^{\T_{\I_a}}\1_{\overline{\B}}((n-m,l)+\X_{1,\K_s+k} ) \right) \geq e^{-\ep}\tfrac{2}{\ga}(1-\ga)^{-A-1}\right],
\]
where $a=\lceil A^2/\ep^2\rceil$, $({\T}_i)=({\T}_i(j',l'))$, $\I_a=\I_a(j',l')$ and $(j',l')=(n-m,l)+\X_{1,\K_s}$. Lemma \ref{lem:manytriangles} and Markov's inequality implies that 
\[
\p{F} \leq \tfrac{\ga}{2}(1-\ga)^{A+1}.
\]
To prove \eqref{ineq:Qgoal3} it will now suffice to show the deterministic claim
\begin{align}\label{ineq:goalfinal}
\text{outside } F  \text{ and } E \quad \Rightarrow \quad \sum_{k=1}^P \1_{\overline{\B}}((n-m,l)+\X_{1,\K_s+k}) > \frac{AB}{\ep}.
\end{align}
Note first that if we are outside $F$ and $\I_a= a$ then (by taking the $\log$)
\begin{align}\label{ineq:counterbalance}
\sum_{k=1}^{\T_{a}}\1_{\overline{\B}}((n-m,l)+\X_{1,\K_s+k} ) \geq  \frac{A^2}{\ep} + O(A)
\end{align}
for large $A$ independent of $\ep$. Assume by contradiction that \eqref{ineq:goalfinal} is false, that is, 
\begin{align}\label{ineq:contraassump}
\sum_{k=1}^P \1_{\overline{\B}}((n-m,l)+\X_{1,\K_s+k}) \leq  \frac{AB}{\ep}
\end{align}
with positive probability outside $F$ and $E$. Let $(j',l')=(n-m,l) + \X_{1,\K_s}$. If $P>B A/\ep$ then there is a first $k\leq P$ such that $(n-m,l) + \X_{1,\K_s+k}\in \D_1=\D(j_{\D_1},l_{\D_1},s_{\D_1})$, hence $\T_1(j',l')$ exists. Since we are outside $E$ we must have $s_{\D_1}< 4^A k^3$, and since $(n-m+\J_{1,\K_s+k}-j_{\D_1}) + \eta(\l_{\D_1}-l-\L_{1,\K_s+k})\leq s_{\D_1} < 4^A k^3$ we have $l+\L_{1,\K_s+k'}>\l_{\D_1}$ for $k' \geq k+4^Ak^3/\eta$, hence $(n-m,l) + \X_{1,\K_s+k'}$ exits $\D_1$. If $P > B A/\ep + k+4^Ak^3/\eta$ then by \eqref{ineq:contraassump} there is a first $k' \leq B A/\ep + k+4^Ak^3/\eta$ such that $(n-m,l)+\X_{1,\K_s+k'} \in \D_2$. We deduce $B A/\ep + k+4^Ak^3/\eta \geq k'=\T_2(j',l') \geq k+4^Ak^3/\eta$. We can keep this iterative procedure and conclude that if $P\geq P_{A,\ep}$ then $\T_a(j',l')$ exists, $\T_a(j',l')\leq P$ and $\I(j',l')\geq a$. However, \eqref{ineq:contraassump} and \eqref{ineq:counterbalance} together imply that $AB/\ep \geq A^2/\ep+O(A)$, which is absurd if $A$ is sufficiently large.
\medskip

\noindent {\bf Step 4.} We now prove Lemma \ref{lem:largetriangles}. We can assume $s'\gtrsim A^2 k $ otherwise there is nothing to prove. Similarly to \eqref{ineq:tailofJKs}, we can use the same splitting to show
\[
\p{l+\L_{1,\K_s+k} \geq l_\D + A^2k} \ls e^{-cA^2k}
\]
Using the bounds \eqref{ineq:decayfirsttime2} and  \eqref{ineq:decayX} is also not hard to deduce that
\begin{align*}
\p{|\J_{1,\K_s+k}-\tfrac{s}{\omega}|  \geq s^{2/3}}  &\leq \p{|\J_{1,\K_s}-\tfrac{s}{\omega}|\geq \tfrac12s^{2/3}} + \p{\J_{\K_s+1,\K_s+k}\geq \tfrac12s^{2/3}} \\ 
& \ls e^{-c s^{1/3}} + e^{-c s^{2/3}} \\
& \ls A^2k/s'
\end{align*}
where in the last inequality we have used that $s\geq m/\log^2 m$ and $s' \leq m^{1/3}$. We obtain
\[
\p{E_{k,s'}} \ls A^2k/s' + e^{-cA^2k} + \p{E_{k,s'} \cap E_*},
\]
where $E_*= (l+\L_{1,\K_s+k} < l_\D + A^2k) \wedge (|\J_{1,\K_s+k}-\tfrac{s}{\omega}|  < s^{2/3})$. Let $J_*\subset \N$ denote the set of $j_{\D'}$ such that under the event $E_*$ the random walk $(n-m,l)+\X_{1,\K_s+k}$ encounters a triangle $\D'=\D(j_{\D'},l_{\D'},s_{\D'})$ (necessarily different from $\D$ since $s=l_\D-l$) and with $s_{\D'}\geq s'$. We claim that $J_*$ is a $c s'$-separated set and that $n-m+\J_{1,\K_s+k}=j_{\D'}+O(A^2k)$. If the claim is true then
\begin{align*}
\p{E_{k,s'} \wedge E_*} & \leq \sum_{\substack{j_{\D'} \in \J_* \\ |j'-  j_{\D'}|\leq A^2k/c}} \p{n-m+\J_{1,\K_s+k} = j'}\\
& \ls s^{-1/2}\sum_{\substack{j_{\D'} \in \J_* \\ |j'-  j_{\D'}|\leq A^2k/c}} G_{s+1}(c(j'-n+m-s/\omega)) \\
& \ls  \frac{A^2 k}{s'}
\end{align*}
where in the second inequality we applied \eqref{ineq:decayfirsttime}, and this would finish the proof of the lemma. 

To prove the remaining claim note first that if under $E_*$ we have $(j',l')=(n-m,l)+\X_{1,\K_s+k} \in \D'$ with $s_{\D'} \geq s'$, then $l'< l_{\D}+A^2k$ and $j'=n-m+s/\omega+O(s^{2/3})$. By the definition of $\D'$ we obtain
\begin{align*}
s_{\D'} \geq l_{\D'}-l'+  \eta(j' -j_{\D'}) > l_{\D'}-l_\D+  \eta(j' - A^2k/\eta-j_{\D'}).
\end{align*}
Assume now that $j'\geq j_{\D'}+A^2k$. We can then let $j''=j'-A^2k$ and conclude that $(j'',l_\D)\in \D'$. However, since $(n-m,l)\in \D$ and $s=l_\D-l$ we have
\begin{align*}
\eta(j''-j_\D) & = \eta(n-m-j_\D) + \eta(j''-(n-m))  \\
& < s_{\D}-s + \eta(j'-(n-m)) \\
& = s_{\D}-(1-\eta/\omega)s +O(s^{2/3})\\
& \leq s_{\D}
\end{align*}
for $m$ sufficiently large, where in the last inequality we have used that $\eta<\omega$ and $s\geq m/\log^2 m$. We conclude $(j'',l_\D)\in \D$, which is absurd since $\D'$ is different from $\D$. We deduce that $j'< j_{\D'}+A^2k$. To finish it remains to show that $J_*$ is $cs'$-separated. Note first that that since $(j',l')\in\D'$ then $l_{\D'}-s_{\D'} \leq l' \leq l_\D+A^2k$. Suppose by contradiction that $l_{\D'}<l_\D+s_{\D'}-\eta Ak^2$. Then $\eta(j'-j_{\D'})+l_{\D'}-l_\D < s_{\D'}$ and so $(j',l_\D)\in \D'$. On the other hand we have 
\[
\eta(j'-j_\D) =  \eta(n-m-j_\D)+\eta s/\omega+O(s^{2/3}) \leq s_\D+ (\eta /\omega-1)s+O(s^{2/3}) \leq s_\D
\]
for large $m$, and so $(j',l_\D)\in \D$, which is absurd.  We deduce that 
\begin{align}\label{ineq:lDprimebounds}
l_\D+s_{\D'}-\eta A^2k \leq l_{\D'} \leq l_\D+s_{\D'}+ A^2k
\end{align}
Let now $\D',\D''$ be two distinct triangles with $s_{\D'},s_{\D''} \geq s'$ such that under the event $E_*$ the random walk $(j'.l')=(n-m,l)+\X_{1,\K_s+k}$ encounters (with non-zero probability). We can assume $j_{\D''}\geq j_{\D'}$. Let $l_*=l_\D+\lfloor s'/2 \rfloor$ and note that because of \eqref{ineq:lDprimebounds} we have $l_{\D'}-s_{\D'} \leq l_* \leq l_{\D'}$, and the for $\D''$, that is, $\l_{\D''}-s_{\D''} \leq l_* \leq l_{\D''}$. Since $\D'$ and $\D''$ are distinct the sets $\{j_{\D'},\ldots,j_{\D'}+(s_{\D'}+l_*-l_\D')/\eta\}$ and $\{j_{\D''},\ldots,j_{\D''}+(s_{\D''}+l_*-l_\D'')/\eta\}$ must be disjoint and so
\[
j_{\D''} \geq j_{\D'}+(s_{\D'}+l_*-l_\D')/\eta \geq j_{\D'} + s'/2 + O(A^2k) \gtrsim j_{\D'}+ s'.
\]
This finishes the proof of the lemma.
\medskip

\noindent {\bf Step 5.} We now prove Lemma \ref{lem:manytriangles} by induction on $a$. If $a=1$ there is nothing to show. Let $Z((j,l),a)$ be the left hand side in \eqref{ineq:manytriangles} and assume $Z((j,l),a-1) \leq e^{\ep}$ for some $a\geq 2$ and all $(j,l)\in \N^2$.  Let $\T_*=\T_*(j,l)$ be the first time where $l+\L_{1,\T_*}>\l_{\D_1}$ and note that if $\I\neq 0$ then $\T_1<\T_*\leq \T_2$. Hence $\I(j,l)-1=\I((j,l)+\X_{1,\T_*})$ and so $\I_a(j,l)-1=\I_{a-1}((j,l)+\X_{1,\T_*})$. We obtain
\begin{align*}
& \e{\exp\left(\ep \I_{a}-\sum_{k=1}^{\T_{\I_{a}}}\1_{\overline{\B}}((j,l)+\X_{1,k} )+\ep\right)} \\ 
& = \p{\I=0} + \p{\I\neq 0}\e{ \exp\left(-\sum_{k=1}^{\T_*}\1_{\overline{\B}}((j,l)+\X_{1,k})+\ep  \right) Z((j,l)+\X_{1,\T_*},a-1)}\\
&\leq \p{\I=0} + \p{\I\neq 0}e^{2\ep}\e{e^{-\1_{\overline{\B}}((j,l)+\X_{1,\T_*})}},
\end{align*}
where in the last inequality we have used the induction hypothesis. We then have to show that $\e{e^{-\1_{\overline{\B}}((j,l)+\X_{1,\T_*})}} \leq \p{\T_*\neq 0}e^{-\ep}$, which is implied by 
\[
\p{\T_* \neq 0 \wedge (j,l) + \X_{a,\T_*} \notin \B} \gtrsim \, \p{\T_*\neq 0}
\]
and selecting $\ep$ sufficiently small. However, an argument analogous to the one employed to show \eqref{ineq:lowerboundnotinB} can be used here to finish the the proof. \hfill \qed

\section{Numerics}\label{numerics}

In this section we explore Conjecture \ref{conj:main} numerically. Recall Definition \ref{def:CandS} and the conditions stated by Theorem \ref{thm:main}:
\begin{enumerate}
\item[(a)] $p$ and $q$ are coprime; 
\item[(b)] $q^{p-1}<p^{{p}}$; 
\item[(c)] $qj+r(j) \equiv 0 \ ({\rm mod}\ p)$.
\end{enumerate}
We would like to investigate the logical relations between conditions (a), (b) and (c) and the statement: 
\begin{enumerate}
\item[(d)] There exists $B>0$ such that for every $N$ there exists $k\geq 0$ such that $C^k(N)\leq B$. 
\end{enumerate}
Note that an equivalent formulation of (d) is that $C$ has only finitely many cycles and that every orbit of $C$ eventually enters in one of these cycles.

\subsection{(d) appears to not imply (a)} Consider the map $C$ with $p=4$, $q=6$ and $r=[26, 20, 18]$. It satisfies conditions (b) and (c) but obviously does not satisfy (a), nevertheless it seems that $C$ has only $3$ cycles
\begin{align*}
& [2, 32, 8] \\ & [38, 248, 62, 392, 98, 608, 152] \\ & [119, 732, 183, 1116, 279, 1692, 423, 2556, 639, 3852,  963, 5796, 1449, 8720, 2180, 545, 3296, \\ & \, \, 824, 206, 1256, 314, 1904, 476]
\end{align*}
and every orbit eventually enters in one of these loops. Indeed, numerically we have confirmed that for all $N\leq 10^{8}$ we have $C^k(N)\in \{2,38,119\}$ for some $k$. On the other hand, the map $C$ with $p=4$, $q=6$ and $r=[-2, -8, -14]$ seems to contain mostly unbounded orbits. Indeed, it seems that the set of inputs with bounded orbits is exactly in $\{N 4^k : 1\leq N \leq 8, k\geq 0\}$ (which has density zero) and that the first unbounded orbit starts with $N=9$, because $C^{10^4}(9)$ has $882$ digits! We found numerically many other examples of maps that seem to exhibit such dichotomy, another one being $p=12$, $q=15$ with $r=[69, 42, 27, -24, 69, 54, -9, 36, -15, -6, -21]$ or $r'=[9, 6, 3, 0, 9, 6, 3, 0, 9, 6, 3]$. The first map $C$ with vector $r$ appears to only have bounded orbits and only two cycles with minimal elements $1$ and $226$ respectively, while for the second map $C$ with $r'$, around $66\%$ of the numbers $N\leq 10^5$ satisfy the property that $C^k(N)>10^{50}$ for some $k\leq 10^4$. Moreover, $N=51$ seems to be the first with an unbounded orbit because $C^{10^5}(51)$ has $1055$ digits!

\subsection{(d) does not imply (b)} Observe that if $p^2$ divides all the numbers $$\{qN+r(N \mod p) : N\in [p^2]\setminus p\N\}$$ then $C^3(N) \leq q N /p^2 + A$, for some $A$, and thus if $q<p^2$ we obtain that $C_{\min}(N) \leq \frac{AN}{1-qp^{-2}}$  for all $N\in \N$. One simple example is the map $C$ associated with $p=3$, $q=6$ and $r=[3,6]$. This map indeed violates (b), however it satisfies that $C^3(N) \leq (2/3)N + 6$ and so $C_{\min} (N) \leq 18$. Indeed, it is easy to show that all orbits eventually enter in one of the loops $[1, 9, 3]$ or $[2, 18, 6]$.

\subsection{(d) appears to not imply (c)} One example is the map $C$ with $p=4$, $q=6$ and $r=[2, 4, 11]$. This map does not satisfies (c) since $6\times 3 + 11=29$, which is not divisible by $4$. However it seems that all orbits are bounded and eventually enter in one of the loops $[1, 8, 2, 16, 4]$ or $[11, 77, 464, 116, 29, 176, 44]$.

\subsection{(a), (b), (c) appears to imply (d)}
We have performed a computer search for all $q< p^{p/(p-1)}$ with $2\leq p  \leq q \leq 10^2$ and with several randomly selected vectors $r$ such that (a), (b), (c) were satisfied. In all these maps we have searched for orbits starting with $N\leq 10^5$ and such that $C^k(N)\geq 10^{10}$ for some $k\leq 10^3$. We could not find a single map $C$ satisfying this condition. As numerical experiments point out, if a maps $C$ has an unbouded orbit, we usually can find one starting with a relatively small $N$, for instance $N\leq 10^5$ for small $p,q,r$. This gives some evidence towards Conjecture \ref{conj:main}.





\begin{thebibliography}{100}

\bibitem{A79}
J. Allouche,
\emph{Sur la conjecture de Syracuse-Kakutani-Collatz}
S\'eminaire de Th\'eorie des Nombres de Bordeaux, (1978--1979), pp. 1--15

\bibitem{BMS}
V. Bergelson, M. Misiurewicz and S. Senti, \emph{Affine Actions of a Free Semigroup on the Real Line}, Ergodic Theory and Dynamical Systems 26 (2006), 1285--1305.

\bibitem{C72}
{J. H. Conway}, 
\emph{Unpredictable iterations}, Number Theory Conference, University of Colorado, Boulder, 1972, 49--52 (1972).

\bibitem{C87}
{J. H. Conway}, 
\emph{FRACTRAN: A Simple Universal Computing Language for Arithmetic}. Open Prob. in Comm. and Comp. (T. M. Cover and B. Gopinath, Eds.) Springer-Verlag, New York 1987, 3--27.

\bibitem{G09}
A. Gut, 
\emph{Renewal Processes and Random Walks. Stopped Random Walks}, 
Springer Series in Operations Research and Financial Engineering. Springer, New York, NY (2009).

\bibitem{H79}
E. Heppner,  
\emph{Eine Bemerkung zum Hasse-Syracuse-Algorithmus,} Arch. Math. (Basel) 31 (1978/79),
317--320.

\bibitem{H74}
J. J. Hunter,
\emph{Renewal Theory in Two Dimensions: Basic Results},
Advances in Applied Probability 6 (1974), no. 2, 376--391.

\bibitem{J}
S. Janson,
\emph{Tail bounds for sums of geometric and exponential variables,}
Statistics \& Probability Letters 135 (2018), 1--6.

\bibitem{KL}
A. Kontorovich, J. Lagarias, \emph{Stochastic models for the $3x + 1$ and $5x + 1$ problems and related problems}, The ultimate challenge: the $3x + 1$ problem, 131--188, Amer. Math. Soc., Providence,
RI, 2010.

\bibitem{KM05} 
{A. Kontorovich and S. J. Miller},
\emph{Benford’s law, values of L-functions and the $3x + 1$} problem, Acta Arith. 120 (2005), no. 3, 269--297.


\bibitem{K94}
I. Korec, 
\emph{A density estimate for the $3x + 1$ problem}, 
Math. Slovaca 44 (1994), no. 1, 85--89.

\bibitem{KL03}
{I. Krasikov and J. Lagarias}, 
\emph{Bounds for the $3x + 1$ problem using difference inequalities}, 
Acta Arith. 09 (2003), 237–258.

\bibitem{L85}
{J. Lagarias}, 
\emph{The $3x+1$ problem and its generalizations}, 
Amer. Math. Monthly 92 (1985), no. 1,
3--23.

\bibitem{L10}
J. Lagarias,
\emph{The Ultimate Challenge: The $3x + 1$}, 
Problem. Amer. Math. Soc, Providence, RI 2010.

\bibitem{LS06}
{J. Lagarias and K. Soundararajan}, \emph{Benford’s law for the $3x + 1$ function}, 
J. London Math. Soc. (2) 74 (2006), no. 2, 289-303.



\bibitem{M10}
K. R. Matthews,
\emph{Generalized 3x+1 mappings: Markov chains and ergodic theory}, 
The ultimate challenge: the $3x+1$ problem, 79-103, Amer. Math. Soc., Providence, RI, 2010.



\bibitem{MW84}
K. Matthews and A.M. Watts. 
\emph{A generalization of Hasse's generalization of the Syracuse algorithm},
Acta Arithmetica 43 (1) (1984), 167--175.

\bibitem{M78}
{H. M\"oller,} 
\emph{Uber Hasses Verallgemeinerung des Syracuse-Algorithmus (Kakutanis Problem),} 
Acta Arith. 34 (1977/78), 219--226.

\bibitem{S1} 
R. P. Steiner, 
\emph{On the $Qx+1$ Problem, $Q$ odd}, Fibonacci Quarterly 19 (1981), 285--288. 

\bibitem{S2} 
R. P. Steiner, 
\emph{On the $Qx + 1$ Problem, Q odd II}, 
Fibonacci Quarterly 19 (1981), 293--296.

\bibitem{T} 
T. Tao,  
\emph{Almost all orbits of the collatz map attain almost bounded values}, 
Forum of Mathematics, Pi. Vol. 10. Cambridge University Press, 2022.

\bibitem{Te}
R. Terras, \emph{A stopping time problem on the positive integers}, Acta Arith. 30 (1976), 241--252.

\bibitem{V}
S. Volkov (2006), \emph{A probabilistic model for the $5k + 1$ problem and related problems}, Stochastic Processes and Applications 116 (2006), 662--674.



\end{thebibliography}
\end{document}